\documentclass[a4,12pt]{article}

\usepackage[T1]{fontenc}
\usepackage[latin1]{inputenc}
\usepackage{amsfonts} 
\usepackage{amssymb}
\usepackage{amsbsy}
\usepackage{eucal}
\usepackage{amsmath}
\usepackage{makeidx}
\usepackage{amsthm}
\usepackage{amsxtra}
\usepackage{marvosym}
\makeindex                       
\usepackage[hang,small,bf,it]{caption}
\usepackage[english]{babel}
\usepackage{latexsym,amssymb}
\usepackage{xcolor}
\usepackage[pdftex]{graphicx}
\usepackage{amsmath,amsfonts}
\usepackage{tikz}

\usepackage{latexsym,amssymb}
\usepackage{xcolor}
\usepackage[pdftex]{graphicx}
\usepackage{amsmath,amsfonts,amsthm}
\usepackage{multicol}
\usepackage{pifont}
\usepackage{geometry}
\usepackage{colortbl}
\usepackage[all]{xy}

\usepackage{tikz}
\usetikzlibrary{matrix,arrows}
\usetikzlibrary{shapes,arrows}

\usepackage{enumitem}

\theoremstyle{plain}
\newtheorem{thm}{Theorem}
\newtheorem{lem}[thm]{Lemma}

\newtheorem{prop}[thm]{Proposition}

\theoremstyle{definition}

\theoremstyle{definition}
\newtheorem{defi}[thm]{Definition}

\newcommand{\R}{\mathbb{R}}
\newcommand{\Z}{\mathbb{Z}}
\newcommand{\N}{\mathbb{N}}
\renewcommand{\H}{\mathbb{H}}
\newcommand{\C}{\mathbb{C}}

\newcommand*{\rom}[1]{\romannumeral}
\renewcommand{\Re}{{Re}}

\DeclareMathOperator{\fl}{flat}
\DeclareMathOperator{\SL}{SL}

\DeclareMathOperator{\GL}{GL}
\DeclareMathOperator{\SO}{SO}

\DeclareMathOperator{\Spin}{Spin}

\DeclareMathOperator{\Tr}{Tr}
\DeclareMathOperator{\tr}{tr}
\DeclareMathOperator{\End}{End}
\DeclareMathOperator{\Vol}{Vol}
\DeclareMathOperator{\rank}{rank}
\DeclareMathOperator{\Id}{Id}
\DeclareMathOperator{\Ad}{Ad}
\DeclareMathOperator{\ad}{ad}
\DeclareMathOperator{\T}{T}
\DeclareMathOperator{\Norm}{Norm}
\DeclareMathOperator{\cent}{Centr}
\DeclareMathOperator{\spec}{spec}
\DeclareMathOperator{\diam}{diam}

\DeclareMathOperator{\prim}{prime}

\DeclareMathOperator{\Ind}{Ind}

\numberwithin{equation}{section}
\numberwithin{thm}{section}

\title{\huge{Ruelle and Selberg zeta functions for non-unitary twists}}

\author{Polyxeni Spilioti}
\date{}
\begin{document}
\maketitle

\textbf{Abstract}. In this paper, we study the Selberg and Ruelle zeta functions on compact hyperbolic odd 
dimensional manifolds. These zeta functions are defined on one complex variable $s$ in some right half-plane of $\C$.
We use the Selberg trace formula for arbitrary not neccesarily unitary
representations of the fundamental group to establish the meromorphic continuation 
of these zeta functions to the whole complex plane.
 
\textit{Keywords}: Selberg zeta function, Ruelle zeta function, Selberg trace formula.

\begin{center}
\section{\textnormal{Introduction}}
\end{center}

The Selberg and Ruelle zeta functions are dynamical zeta functions, which can be associated with the geodesic flow on the unit sphere bundle $S(X)$ of  
a compact hyperbolic riemannian manifold $X$. Namely, they provide information about the lenghts of the closed geodesics, also 
called length spectrum. If we consider the geodesic flow $\phi$ on $S(X)$ and $d\phi$ its differential, then $\phi$
has the Anosov property, i.e.,
there exists a $d\phi_{t}$-invariant continuous splitting 
\begin{equation*}
 TS(X)=E^{s}\oplus E^{c}\oplus E^{u},
\end{equation*}
where $E^{s},E^{u}$ consist of vectors that shrink, respectively expand exponentially,
and $E^{c}$ is the one dimensional subspace of vectors tangent to the flow, with respect to the riemannian metric, as $t\rightarrow\infty$.

These zeta functions can be represented by Euler products, which converge
in some right half-plane of $\C$. The main goal of this paper is to prove the meromorphic continuation of these functions to the whole complex plane. 
The Ruelle zeta function associated with the geodesic flow on the unit sphere bundle of 
a closed manifold with $C^{\omega}$-riemannian metric of negative curvature has been studied
by Fried in \cite{Friedmero}. It is defined by
\[
R(s)=\prod_{\gamma}(1-e^{-sl(\gamma)}),
\]
where $\gamma$ runs over all the prime closed geodesics
and $l(\gamma)$ denotes the length of $\gamma$.
In \cite[Corollary, p.180]{Friedmero}, it is proved that it 
admits a meromorphic continuation to the whole complex plane.

We consider a compact hyperbolic manifold $X$ of odd dimension $d$, obtained as follows.
Let $G=\SO^{0}(d,1)$ and $K=\SO(d)$. Then, $K$ is a maximal compact subgroup of $G$.
Let $\widetilde{X}:=G/K$. $\widetilde{X}$ can be equipped with a $G$-invariant metric, which is unique up to scaling 
and is of constant negative curvature.
If we normalize this metric such that it has constant negative curvature $-1$, then 
$\widetilde{X}$, equipped with this metric, is isometric to $\H^{d}$.
Let $\Gamma$ be a discrete torsion-free subgroup of $G$ such that $\Gamma\backslash G$ is compact. 
Then $\Gamma$ acts by isometries on $\widetilde X$ and $X=\Gamma\backslash \widetilde X$ is a
compact oriented hyperbolic manifold of dimension $d$. 
This is a case of a locally symmetric space of non-compact type of real rank 1. This means that in the Iwasawa
decomposition $G=KAN$, $A$ is a multiplicative torus of dimension 1, i.e.,
$A\cong\R^+$. 

For a given $\gamma\in\Gamma$ we denote by $[\gamma]$ the $\Gamma$-conjugacy class
of $\gamma$. If $\gamma\neq e$, then there is a unique closed geodesic 
$c_{\gamma}$ associated with $[\gamma]$. Let $l(\gamma)$ denote the length of
$c_{\gamma}$.
The conjugacy class $[\gamma]$ is called prime if 
there exist no $k\in\N$ with $k>1$ and $\gamma_{0}\in\Gamma$ such that $\gamma=\gamma_{0}^{k}$.
The prime geodesics correspond to the prime conjugacy classes and are those 
geodesics that trace out their image exactly once. 

Let $M$ be the centralizer of $A$ in $K$. We define the zeta functions associated
with unitary irreducible representations $\sigma$ of $M$ and 
finite dimensional representations $\chi$ of $\Gamma$.
Let $\chi\colon\Gamma\to \GL(V_{\chi})$ be a finite dimensional representation of $\Gamma$. Let
$\sigma\in\widehat M$. 
Then the twisted Selberg zeta function $Z(s;\sigma,\chi)$ is defined by the infinite product
\begin{equation*}
Z(s;\sigma,\chi):=\prod_{\substack{[\gamma]\neq e,\\ [\gamma]\prim}} \prod_{k=0}^{\infty}\det\Big(\Id-\big(\chi(\gamma)\otimes\sigma(m_\gamma)\otimes S^k(\Ad(m_\gamma a_\gamma)_{\overline{\mathfrak{n}}})\big)e^{-(s+|\rho|)l(\gamma)}\Big),
\end{equation*}
where $s\in \C$, $\overline{\mathfrak{n}}$ is the sum of the negative root spaces of the system $(\mathfrak{g},\mathfrak{a})$
and $S^k(\Ad(m_\gamma a_\gamma)_{\overline{\mathfrak{n}}})$ denotes the $k$-th
symmetric power of the adjoint map $\Ad(m_\gamma a_\gamma)$ restricted to $\overline{\mathfrak{n}}$.

The twisted Ruelle zeta function $R(s;\sigma,\chi)$ is defined by the infinite product
\begin{equation*}
 R(s;\sigma,\chi):=\prod_{\substack{[\gamma]\neq{e}\\ [\gamma]\prim}}\det(\Id-\chi(\gamma)\otimes\sigma(m_{\gamma})e^{-sl(\gamma)})^{(-1)^{d-1}}.
\end{equation*}
Both $Z(s;\sigma,\chi)$, $R(s;\sigma,\chi)$ converge absolutely and uniformly on compact subsets of some half-plane of $\C$. 

In our case, where $X=\Gamma/\H^{d}$, 
the dynamical zeta functions 
are twisted by a representation $\chi$ of $\Gamma$.
For unitary representations of $\Gamma$,
these zeta functions have been studied by Fried (\cite{Fried}) and Bunke and Olbrich (\cite{BO}). 
In \cite[Theorem 1]{Fried}, Fried proved that the Ruelle zeta function on a closed oriented hyperbolic manifold associated with an
acyclic orthogonal representation of the fundamental group admits a meromorphic continuation to the whole complex plane 
and furthermore the 
absolute value of the Ruelle zeta function 
evaluated at zero equals the Ray-Singer
analytic torsion as it is introduced in \cite{RS}. Apart from the fact that this is of great importance, since it 
provides a connection between the geometry (geodesic flow expressed by the Ruelle zeta function) of the manifold and a 
spectral invariant (analytic torsion), the theorem of Fried gave rise to other applications in the field of spectral geometry,
see for example \cite{MMar}, \cite{M2}.
In \cite{M2}, M\"{u}ller considered the asymptotic behavior of the analytic torsion of a closed hyperbolic $3$-manifold $X$,
associated with the $m$-th symmetric power of the standard representation $\tau_{m}$ of the group $\SL(2,\C)$.
In particular, he proved that for a closed hyperbolic oriented $3$-manifold,
\begin{equation*}
 -\log T_{X}(\tau_{m})=\frac{\Vol(X)}{4\pi}m^2+O(m),
\end{equation*}
as $m\rightarrow \infty$ (\cite[Theorem 1.1]{M2}).
In order to prove this theorem, he used the expression of the analytic torsion in terms of the Ruelle zeta function attached 
to $\tau_{m}$ and based on the results of Wotzke (\cite{Wo}). 
In his thesis, Wotzke generalized the theorem 
of Fried to the case of the induced representation $\tau|_{\Gamma}$ of $\Gamma$
arising from a restriction of a finite dimensional complex representation $\tau\colon G\rightarrow \GL(V)$ of $G$. 
By \cite[Proposition 3.1]{MM}, there exists an isomorphism between the locally homogenous vector bundle $E_{\tau}$ over $X$ 
associated with $\tau|_{K}$  and the flat vector bundle $E_{\fl}$ over $X$
associated with $\tau|_{\Gamma}$. i.e.,
\begin{equation*}
 \Gamma\backslash(G/K\times V)\cong (\Gamma\backslash G\times V)/K.
\end{equation*}
Once this isomorphism is obtained, a hermitian 
fiber metric in $E_{\tau}$ (\cite[Lemma 3.1]{MM}) descends to a fiber metric in $E_{\fl}$. Therefore, the Laplace operator
associated with $\tau|_{\Gamma}$ is a formally self-adjoint operator and the whole harmonic analysis on locally symmetric 
spaces is provided.

On the other hand, Bunke and Olbrich in \cite{BO} consider a unitary representation of the fundamental group 
for all the cases of compact locally symmetric spaces of real rank 1 and of non-compact type, i.e.,
the compact manifolds whose universal coverings are the real, complex, or quaternionic  hyperbolic space, or the Cayley plane.
Using the Selberg trace formula as main tool for wave operators induced by certain Laplace-type and Dirac operators, they proved that the
Selberg and Ruelle zeta functions admit a meromorphic continuation to the whole complex and furthermore satisfy 
functional equations.

Our main aim is to generalize the results of Bunke and Olbrich to the case of non-unitary representations of 
the fundamental group.
Contrary to the setting of Wotzke, we can no longer use the results in \cite{MM},
since we treat the case of an arbitrary finite dimensional representation $\chi\colon\Gamma\to \GL(V_{\chi})$ of $\Gamma$. 
Let $E_{\chi}$ be the associated to $\chi$ flat vector bundle over $X$, equipped with a flat connection $\nabla^{\chi}$.
In general, there is no hermitian metric $h^{\chi}$, which is compatible with the 
flat connection. To overcome this problem we use the flat Laplacian, first introduced by M\"{u}ller in \cite{M1}.
In fact, we use a more general operator, the twisted Bochner-Laplace operator defined as follows.

Let $\tau\colon K\rightarrow \GL(V_{\tau})$ be a complex finite dimensional unitary representation of $K$. Let $\widetilde{E}_{\tau}:=G\times_{\tau}V_{\tau}\rightarrow \widetilde{X}$ 
be the associated homogenous vector bundle over $\widetilde{X}$.
Let $E_{\tau}:=\Gamma\backslash(G\times_{\tau}V_{\tau})\rightarrow X$ be the locally homogenous vector bundle over $X$.
Let $\Delta_{\tau}$ be the Bochner-Laplace operator associated with $\tau$.
We define the operator $\Delta^{\sharp}_{\tau,\chi}$ acting on $C^{\infty}(X,E_{\tau}\otimes E_{\chi})$.
Locally, it can be described as
\begin{equation*}
 \widetilde{\Delta}^{\sharp}_{\tau,\chi}=\widetilde{\Delta}_{\tau}\otimes\Id_{V_{\chi}}, 
\end{equation*}
where $\widetilde{\Delta}^{\sharp}_{\tau,\chi}$ and $\widetilde{\Delta}_{\tau}$  
are the lifts to $\widetilde{X}$ of 
$\Delta^{\sharp}_{\tau,\chi}$ and $\Delta_{\tau}$, respectively. 
Our operator is not self-adjoint anymore.
However, it has the same principal symbol as $\Delta_{\tau}$ 
and hence has nice spectral properties, i.e., the spectrum of 
${\Delta}^{\sharp}_{\tau,\chi}$ is a discrete subset of a positive cone in $\C$. 
We consider the corresponding heat semi-group $e^{-t\Delta^{\sharp}_{\tau,\chi}}$ acting on the space
of smooth sections of the vector bundle $E_{\tau}\otimes E_{\chi}$.
By \cite[Lemma 2.4]{M1}, it is an integral operator with smooth kernel.
Hence, we can consider the trace of the
operator $e^{-t\Delta^{\sharp}_{\tau,\chi}}$ and derive a corresponding trace formula. 

We already associated the Selberg and Ruelle zeta functions with irreducible representations $\sigma$ of $M$. These representations are chosen precisely to be
the representations arising from restrictions of representations of $K$.
The keypoint of the proof of the meromorphic continuation of the Selberg zeta function is the Selebrg trace formula for the operator
 $e^{-tA_{\tau,\chi}^{\sharp}(\sigma)}$, 
 where $A_{\tau,\chi}^{\sharp}(\sigma)$ is the twisted differential operator associated with 
 $\sigma\in \widehat{M}$ and induced by $\Delta^{\sharp}_{\tau,\chi}$.

\begin{thm}
	For every $\sigma \in \widehat{M}$ we have
	\begin{align*} \label{f:trace formula 1}
		\Tr(e^{-tA_{\chi}^{\sharp}(\sigma)})=&\dim(V_{\chi})\Vol(X)\int_{\R}e^{-t\lambda^{2}}P_{\sigma}(i\lambda)d\lambda\\
		&+\sum_{[\gamma]\neq e} \frac{l(\gamma)}{n_{\Gamma}(\gamma)}L_{sym}(\gamma;\sigma)
		\frac{e^{-l(\gamma)^{{2}}/4t}}{(4\pi t)^{1/2}}, 
	\end{align*} 
	where 
	\begin{equation*}
		L_{sym}(\gamma;\sigma)= \frac{\tr(\sigma(m_{\gamma})\otimes\chi(\gamma))e^{-|\rho|l(\gamma)}}{\det(\Id-\Ad(m_{\gamma}a_{\gamma}))_{\overline{\mathfrak{n}}}}.
	\end{equation*}
\end{thm}
Our main results are stated in the following theorems.
\begin{thm}
	The Selberg zeta function $Z(s;\sigma,\chi)$ admits a meromorphic continuation to the whole complex plane $\C$. The set of the singularities
	equals $\{s_{k}^{\pm}=\pm i \sqrt{t_{k}}:t_{k}\in \spec(A^{\sharp}_{\chi}(\sigma)), k\in\N\}$.
	The orders of the singularities are equal to $m(t_{k})$, where $m(t_{k})\in\N$ denotes the algebraic multiplicity of the eigenvalue $t_{k}$.
	For $t_{0}=0$, the order of the singularity $s_{0}$ is equal to $2m(0)$.
\end{thm}
\begin{thm}
	For every $\sigma\in\widehat{M}$, the Ruelle zeta function $R(s;\sigma,\chi)$ admits a meromorphic continuation to the whole complex plane $\C$.
\end{thm}


As it is mentioned before, the
operator $A^{\sharp}_{\chi}(\sigma)$ is induced by the twisted Bochner-Laplace operator and therefore has similar nice spectral properties, i.e.,
its spectrum is discrete an contained in a translate of a positive cone in $\C$.
By Theorem 1.2, we observe that the singularities of the Selberg zeta function are located precisely at the points 
$t_{k}$, which belong to the discrete spectrum of the twisted operator $A^{\sharp}_{\chi}(\sigma)$. 
Hence, this theorem provides an interesting relation between the dynamics of hyperbolic manifolds and the spectral theory 
of the twisted Bochner-Laplace operators.

Basic notions and facts from representation theory and the theory of locally symmetric spaces are presented in section 2.
In section 3, we introduce the Selberg and Ruelle zeta functions associated with the geodesic flow of a compact hyperbolic manifold 
and prove that they converge in some right half-plane of $\C$ for an arbitrary finite dimensional representation of $\Gamma$.
Section 4 is devoted to the study of the twisted Bochner-Laplace operator on a compact hyperbolic odd dimensional
manifold $X$ and its spectral properties.
Next, in section 5, we obtain a trace formula for the heat operator, induced by specific twisted Bochner-Laplace operators.
Finally, in section 6, we provide the proof of the meromorphic continuation of the dynamical zeta functions to the whole complex plane,
as it is stated in Theorems 1.2 and 1.3.

\textbf{Acknowledgement.} The present paper is part of the author's PhD thesis and therefore she is 
grateful to her supervisor, Werner M\"{u}ller, for his helpful suggestions, remarks and advices.

\begin{center}
\section{\textnormal{Preliminaries}}
\end{center}

A special case of a locally symmetric space of real rank 1 is a compact hyperbolic locally symmetric manifold with universal 
covering the real hyperbolic space
\begin{equation*}
\H^{d}=\{(x_{1},\ldots,x_{d+1})\in\R^{d+1}:x_{1}^{2}-x_{2}^{2}\ldots-x_{d+1}^{2}=1,x_{1}>0\},
\end{equation*}
where $d=2n+1$, and $n\in\N_{>0}$ is an odd integer.
We consider the universal coverings $G=\Spin(d,1)$ of $\mathrm{SO}^0(d,1)$ and $K=\Spin(d)$ of $\mathrm{SO}(d)$, respectively.
We set  $\widetilde{X}:=G/K$. 
Let $\mathfrak{g},\mathfrak{k}$ be the Lie algebras of $G$ and $K$, respectively. Let $ \mathfrak{g}=\mathfrak{k}\oplus\mathfrak{p}$
be the Cartan decomposition of $\mathfrak{g}$.
We denote by $\Theta$ the Cartan involution of $G$ and $\theta$ be the differential of $\Theta$ at $e_{G}=e$, which is the identity element of $G$.
Let $\mathfrak{a}$ be a Cartan subalgebra of $\mathfrak{p}$, i.e., a maximal abelian subalgebra of $\mathfrak{p}$.
There exists a canonical isomorphism $T_{eK}\cong \mathfrak{p}$.
We consider the subgroup $A$ of $G$ with Lie algebra $\mathfrak{a}$. Let $M:=\cent_{K}(A)$ be the centralizer of $A$ in $K$. 
Then, $M=\Spin(d-1)$ or $M=\SO(d-1)$. Let $\mathfrak{m}$ be its Lie algebra
and $\mathfrak{b}$ a Cartan subalgebra of $\mathfrak{m}$. Let $\mathfrak{h}$ be a Cartan subalgebra of $\mathfrak{g}$.
We consider the complexifications $\mathfrak{g}_{\C}:=\mathfrak{g}\oplus i\mathfrak{g}$, $\mathfrak{h}_{\C}:=\mathfrak{h}\oplus i\mathfrak{h}$
and $\mathfrak{m}_{\C}:=\mathfrak{m}\oplus i\mathfrak{m}$.
Let $B(X,Y)$ be the Killing form on $\mathfrak{g}\times \mathfrak{g}$ defined by $ B(X,Y)=\Tr (\ad(X)\circ \ad(Y))$.
It is a symmetric bilinear from. 
We consider the inner product 
\begin{equation}
 \langle Y_1, Y_2\rangle_{0}:=\frac{1}{2(d-1)}B(Y_1,Y_2),\quad Y_1,Y_2 \in \mathfrak{g},
\end{equation}
induced by the Killing form.
The restriction of $\langle\cdot,\cdot\rangle_{0}$ to $\mathfrak{p}$ defines
an inner product on $\mathfrak{p}$ and
hence induces a $G$-invariant riemannian metric on $\widetilde{X}$, which 
has constant curvature $-1$. Then, $\widetilde{X}$, equipped with this metric,
is isometric to $\H^{d}$.
Let $\Gamma\subset G$ be a lattice, i.e., a discrete subgroup of $G$ such that $\Vol(\Gamma\backslash G)<\infty$.
$\Gamma$ acts properly discontinuously on $\widetilde{X}$ and 
$X:=\Gamma\backslash \widetilde{X}$
is a locally symmetric space of finite volume.
We assume that $\Gamma$ is torsion free. i.e.,
there exists no $\gamma\in\Gamma$ with $\gamma\neq e$ such that for $k=2,3,\ldots$, $\gamma^{k}=e$.
Then, $X$ is a locally symmetric manifold. If in addition $\Gamma$ is cocompact, then $X$ is a locally symmetric compact 
hyperbolic manifold of odd dimension $d$.

Let $G=KAN$
be the standard Iwasawa decomposition of $G$.
Let $\Delta^{+}(\mathfrak{g},\mathfrak{a})$ be the set of positive roots of the system $(\mathfrak{g},\mathfrak{a})$.
Then, $\Delta^{+}(\mathfrak{g},\mathfrak{a})=\{\alpha\}$. Let  $M'=\Norm_{K}(A)$ be the normalizer of $A$ in $K$.
We define the restricted Weyl group as the quotient $ W_{A}:=M'/M$.
Then, $W_{A}$ has order 2. Let $w\in W_{A}$ be a non-trivial element of $W_{A}$, and 
$m_{w}$ a representative of $w$ in $M'$.
The action of $W_{A}$ on $\widehat{M}$ is defined by
\begin{equation*}
(w\sigma)(m):=\sigma(m_{w}^{-1}mm_{w}),\quad m\in M, \sigma\in\widehat{M}.
\end{equation*}
Let $H_{\R}\in\mathfrak{a}$ such that $\alpha(H_{\R})=1$. With respect to the inner product (2.1), $H_{\R}$ has norm 1.
We define 
\begin{equation}
 A^{+}:=\{\exp(tH_{\R})\colon t\in\R^{+}\}.
\end{equation}
We define also 
\begin{align}
 &\rho:=\frac{1}{2}\sum_{\alpha\in \Delta^+(\mathfrak{g},\mathfrak{a})}\dim(\mathfrak{g}_\alpha)\alpha,\\
&\rho_{\mathfrak{m}}:=\frac{1}{2}\sum_{\alpha\in \Delta^+(\mathfrak{m}_{\C},\mathfrak{b})}\alpha.\\\notag
\end{align}
The inclusion $i\colon M\hookrightarrow K$ induces the restriction map $i^{*}\colon R(K)\rightarrow R(M)$, where 
$R(K),R(M)$ are the representation rings over $\Z$ of $K$ and $M$, respectively.
Let $\widehat{K},\widehat{M}$ be the sets of equivalent classes of 
irreducible unitary representations of $K$ and $M$, respectively. Then,
for the highest weight $\nu_{\tau}$ of $\tau\in\widehat{K}$ we have
\begin{align*}
  \nu_{\tau}=(\nu_{1},\ldots,\nu_{n}),
\end{align*}
where $\nu_{1}\geq\ldots\geq\nu_{n}$ and $\nu_{i},i=1,\ldots,n$ are all integers or all half integers
(that is $\nu_{i}=q_{i}/2, q_{i}\in\Z$) and for the highest weight $\nu_{\sigma}$ of $\sigma\in\widehat{M}$,
\begin{align}
  \nu_{\sigma}=(\nu_{1},\ldots,\nu_{n-1},\nu_{n}),
\end{align}
where $\nu_{1}\geq\ldots\geq\nu_{n-1}\geq\lvert\nu_{n}\rvert$ and $\nu_{i},i=1,\ldots,n$ are all integers or all half integers (see  \cite[p. 20]{BO}).
Let $s$ be the spin representation of $K$\index{spin representation}, given by
\begin{equation*}
 s\colon K\rightarrow\End(\Delta_{2n})\oplus\End(\Delta_{2n})\xrightarrow{pr} \End(\Delta_{2n})
\end{equation*}
where $\Delta_{2n}:={\C^{2}}^{k}$ such that $n=k$, and $pr$ denotes the projection onto the first component (\cite[p.14]{Friedb}).
We set for abbreviation $S=\Delta_{2n}$. Let $(s^{+},S^{+})$, $(s^{-},S^{-})$ be the  half spin representations of $M$\index{spin representation!half spin representations},
where $S^{\pm}:=\Delta^{\pm}$ (\cite[p.22]{Friedb}).
The highest weight of $s$ is given by $
 \nu_{s}=(\frac{1}{2},\ldots,\frac{1}{2})$,
and the highest weights of $s^{+},s^{-}$ are
$\nu_{s^{+}}=(\frac{1}{2},\ldots,\frac{1}{2})$, 
$\nu_{s^{-}}=(\frac{1}{2},\ldots,-\frac{1}{2})$,
respectively (\cite[p. 20]{BO}).

We consider now the parametrization of the principal series representation. 
Let $P=MAN$ be the standard parabolic subgroup of $G$. For $(\sigma,V_{\sigma})\in \widehat{M}$,
we define the space $\mathcal{H}_{\sigma}$ of continuous functions on $G$ by
\begin{equation*}
 \mathcal{H}_{\sigma}:=\{f\colon G\rightarrow V_{\sigma}: f(gman)=e^{-(i\lambda+|\rho|)t}\sigma^{-1}(m)f(g), \forall g\in G, \forall man\in P\},
\end{equation*}
where $\lambda\in\C$ and $\rho$ as in (2.3), with norm 
\begin{equation}
 \lVert f \rVert_{c}=\int_{K}\lVert f(k) \rVert^{2}dk.
\end{equation}
We define the principal series representation as the induced representation 
\begin{equation*}
 \pi_{\sigma,\lambda}:=\Ind_{P}^{G}(\sigma\otimes e^{i\lambda}\otimes \Id),
\end{equation*}
with representation space the Hilbert space, obtained by completion of $\mathcal{H}_{\sigma}$ with respect to the norm 
$\lVert \cdot\rVert_{c}$ in (2.6).
For $f\in\mathcal{H}_{\sigma}$, the action of $G$ on  $f$ is given by $ \pi_{\sigma,\lambda}(g)f(g')=f(g^{-1}g')$.
 $\mathfrak{a}_{\C}^{*}$ is the space of the linear functionals on $\mathfrak{a}_{\C}$.
In the definition of the space $ \mathcal{H}_{\sigma}$,
$\lambda$ is a complex number. Hence,  $\mathfrak{a}_{\C}^{*}$ 
is identified with $\C$, using the positive root.
If $\lambda\in\R$, then the representation $\pi_{\sigma,\lambda}$ is unitary. 
In addition, if $\lambda\in \R-\{0\}$, then $\pi_{\sigma,\lambda}$ is irreducible.

Let $\mu_{PL}(\pi_{\sigma,\lambda})$ be the Plancherel measure, viewed as a measure on the set of 
the principal series represenations $\pi_{\sigma,\lambda}$.
Since $\rank(G)>\rank(K)$, by classical result of Harish-Chandra (\cite{HC2}), the set of the 
discrete series representations of $G$ is empty.
Then, by \cite[Theorem 13.2]{Knapp}, 
\begin{equation*}
 d\mu_{PL}(\pi_{\sigma,\lambda})=P_{\sigma}(i\lambda)d\lambda.
\end{equation*}
Here, $P_{\sigma}(i\lambda)$ is the Plancherel polynomial
given by
\begin{equation*}
 P_{\sigma}(i\lambda)=\prod_{\alpha\in\Delta^{+}(\mathfrak{g}_{\C},\mathfrak{h})}\frac{\langle i\lambda+\nu_{\sigma}+\rho_{\mathfrak{m}},\alpha\rangle}{\langle \rho_{\mathfrak{g}},\alpha\rangle},
\end{equation*}
where $\langle\cdot,\cdot\rangle$ is the inner product on $\mathfrak{a}^*$, $\nu_{\sigma}$ is the highest weight of $\sigma$ as in (2.5),
$\rho_{\mathfrak{m}}$ is defined as in (2.4),
and $\rho_{\mathfrak{g}}$ is defined by
\begin{equation*}
 \rho_{\mathfrak{g}}:=\frac{1}{2}\sum_{\alpha\in\Delta^{+}(\mathfrak{g}_{\C},\mathfrak{h})}\alpha,
\end{equation*}
(see \cite[p.46]{BO}).
Let $z=i\lambda\in\C$.
Then, by \cite[p.264-265]{Mi2}, $P_{\sigma}(z)$ is an even polynomial of $z$,
and hence $ P_{\sigma}(z)=P_{\sigma}(-z)$.
\newpage

\begin{center}
{\section{\textnormal{Twisted Selberg and Ruelle zeta functions}}}
\end{center}

We consider the twisted Ruelle and Selberg zeta functions associated with the geodesic flow on the
sphere vector bundle $S(X)$ of $X=\Gamma\backslash G/ K$. Since $K$ acts transitively on the unit vectors of $\mathfrak{p}$, 
by the adjoint representation, $S(\widetilde{X})$ can be represented by the homogenous space $G/M$. 
Therefore, $S(X)=\Gamma\backslash G/M$.

We recall the Cartan decomposition $G=KA^{+}K$ of $G$,
where $A^{+}$ is as in (2.1).
Then, every element $g\in G$ can be written as $g=ha_{+}k$, where $h,k\in K$ and $a_{+}=\exp(tH_{\R})$ for some $t\in \R^{+}$.
The positive real number $t$ equals $ t=d(eK,gK),$
where $d$ denotes the geodesic distance on $\widetilde{X}$.
It is a well known fact (\cite{GKM}) that there is a 1-1 correspondence between the closed geodesics on a manifold $X$ with negative sectional curvature 
and the non-trivial conjugacy classes of the fundamental group $\pi_{1}(X)$ of $X$.
The hyperbolic elements of $\Gamma$ can be realized as the semisimple elements of this group, i.e., the diagonalizable elements of $\Gamma$.
Since $\Gamma$ is a cocompact subgroup of $G$, we realize every element $\gamma\in\Gamma-\{e\}$ as hyperbolic.
We denote by $c_{\gamma}$ the closed geodesic on $X$, associated with the hyperbolic conjugacy class $[\gamma]$.
We denote also by $l(\gamma)$ the length of $c_{\gamma}$.
Since $\Gamma$ is torsion-free, $l(\gamma)$ is always positive and therefore we can obtain an infimum for the length spectrum $\spec(\Gamma):=\{l(\gamma):\gamma\in\Gamma\}$.
An element $\gamma\in\Gamma$ is called primitive if there exists no $n\in\N$ with $n>1$ and $\gamma_{0}\in\Gamma$ such that $\gamma=\gamma_{0}^{n}$.
We associate to a primitive element $\gamma_{0}\in\Gamma$ a prime geodesic on $X$.
The prime geodesics correspond to the periodic orbits of minimal length.
Hence, if a hyperbolic element $\gamma$ in $\Gamma$ is generated by a primitive element $\gamma_{0}$, then
there exists a $n_{\Gamma}(\gamma)\in \N$ such that $\gamma=\gamma_{0}^{n_{\Gamma}(\gamma)}$. The corresponding 
closed geodesic is of length $l(\gamma)=n_{\Gamma}(\gamma)l(\gamma_{0})$.

We lift now the closed geodesic $c_{\gamma}$ to the universal covering $\widetilde{X}$.
For $\gamma\in\Gamma$, $l(\gamma):=\inf \{d(x,\gamma x):x\in\widetilde{X}\}$,
and $l(\gamma)=\inf \{d(eK,g^{-1}\gamma gK):g\in G\}$.
Hence, we see that the length of the closed geodesic $l(\gamma)$ depends only on $\gamma\in\Gamma$.
Let $\gamma \in \Gamma$, with $ \gamma\neq e$ and $\gamma$ hyperbolic. Then, by \cite[Lemma 6.5]{Wa} there exist a $g\in G$, a $m_{\gamma} \in M$, and an 
$a_{\gamma} \in A^{+}$, such that $ g^{-1}\gamma g=m_{\gamma}a_{\gamma}$.
The element $m_{\gamma}$ is determined up to conjugacy classes in $M$, and the element $a_{\gamma}$ depends only on $\gamma$.

Analogous to the consideration of \cite[Section 3.1]{BO},
we define the geodesic flow $\phi$ on $S(X)$ by the map
\begin{equation*}
\phi\colon\R\times S(X)\ni (t,\Gamma gM)\rightarrow \Gamma g\exp(-tH_{\R}) M \in S(X).
\end{equation*}
A closed orbit of $\phi$ is described by the set $c:=\{\Gamma g\exp(-tH_{\R}) M\colon t\in\R\}$,
where $g\in G$ is such that $g^{-1}\gamma g:=m_{\gamma}a_{\gamma}\in MA^{+}$.

The Anosov property of the geodesic flow $\phi$ on $S(X)$ can be expressed by the following 
$d\phi$-invariant splitting of $TS(X)$
\begin{equation}
 TS(X)=T^sS(X)\oplus T^cS(X)\oplus T^uS(X),
\end{equation}
where $T^sS(X)$ consist of vectors that shrink exponentially, $T^uS(X)$ expand exponentially,
and $T^{c}S(X)$ is the one dimensional subspace of vectors tangent to the flow, with respect to the riemannian metric, as $t\rightarrow\infty$.
The spitting in (3.1), corresponds to splitting
\begin{equation}
  TS(X)=\Gamma\backslash G\times_{\Ad}(\overline{\mathfrak{n}}\oplus \mathfrak{a} \oplus \mathfrak{n}),
\end{equation}
where $\Ad$ denotes the adjoint action of $\Ad(\exp(-tH_{\R}))$ on $\overline{\mathfrak{n}}, \mathfrak{a},\mathfrak{n}$, 
and $\overline{\mathfrak{n}}=\theta \mathfrak{n}$ is the sum of the negative root spaces of
the system $(\mathfrak{g},\mathfrak{a})$.

Let $(\sigma,V_\sigma)\in \widehat{M}$ and $(\chi, V_\chi)$ be a finite dimensional representation of $\Gamma$.
Let $E(\sigma,\chi):=\Gamma\backslash (G\times_{\sigma\otimes\chi}(V_\sigma\otimes V_\gamma))\rightarrow S(X)$ be the vector bundle
over $S(X)$, associated to the  representations $\sigma$ and $\chi$. The action of $M$ and $\Gamma$ on $G\times(V_\sigma\otimes V_\gamma)$
is defined by
\begin{equation*}
 [g, v\otimes w]=\{(\gamma g, \sigma(m)v\otimes\chi(\gamma)w)\colon g\in G, v\in V_{\sigma}, w\in {V_{\chi}}, \gamma\in\Gamma, m\in M \}, 
\end{equation*}
We can lift the flow $\varphi$ to the flow $\varphi_{\sigma,\chi}$ on $E(\sigma,\chi)$ by the map
\begin{equation*}
 \phi_{\sigma,\chi}\colon\R \times E(\sigma, \chi)\ni (t, [g,v\otimes w])\mapsto[g\exp(-tH_{\R}), v\otimes w]\in E(\sigma, \chi).
\end{equation*}
\begin{defi}
Let $\chi\colon\Gamma\rightarrow \GL(V_{\chi})$ be a finite dimensional representation of $\Gamma$ and $\sigma\in \widehat{M}$.
The twisted Selberg zeta function $Z(s;\sigma,\chi)$ for 
$X$ is defined  by the infinite product
\begin{equation}
Z(s;\sigma,\chi):=\prod_{\substack{[\gamma]\neq{e}\\ [\gamma]\prim}} \prod_{k=0}^{\infty}\det\big(\Id-(\chi(\gamma)\otimes\sigma(m_\gamma)\otimes S^k(\Ad(m_\gamma a_\gamma)|_{\overline{\mathfrak{n}}})) e^{-(s+|\rho|)\lvert l(\gamma)}\big),
\end{equation}
where $s\in \C$, $\overline{\mathfrak{n}}=\theta \mathfrak{n}$ is the sum of the negative root spaces of $\mathfrak{a}$,
$S^k(\Ad(m_\gamma a_\gamma)_{\overline{\mathfrak{n}}})$ denotes the $k$-th
symmetric power of the adjoint map $\Ad(m_\gamma a_\gamma)$ restricted to $\mathfrak{\overline{n}}$, and $\rho$ is as in (2.3).
\end{defi}
\begin{defi} Let $\chi\colon\Gamma\rightarrow \GL(V_{\chi})$ be a finite dimensional representation of $\Gamma$ and $\sigma\in \widehat{M}$.
The twisted Ruelle zeta function $ R(s;\sigma,\chi)$ for $X$ is defined by the infinite product
\begin{equation}
 R(s;\sigma,\chi):=\prod_{\substack{[\gamma]\neq{e}\\ [\gamma]\prim}}\det\big(\Id-\chi(\gamma)\otimes\sigma(m_{\gamma})e^{-sl(\gamma)}\big)^{(-1)^{d-1}}.
\end{equation}
\end{defi}
In order to prove the convergence of the zeta functions,
we need to find an upper bound for the character of any finite dimensional representation of $\Gamma$.
\begin{lem}
 Let $\chi\colon\Gamma\rightarrow \GL(V_{\chi})$ be a finite dimensional representation of $\Gamma$.
 Then, there exist positive constants $K,k>0$ such that 
 \begin{equation}
  \lvert\tr(\chi(\gamma))\rvert\leq Ke^{kl(\gamma)},\quad \forall\gamma\in\Gamma-\{e\}.
 \end{equation}
\end{lem}
\begin{proof}
 We fix a finite set of generators $L=\{\gamma_1,\ldots,\gamma_{r}\}$ of $\Gamma$ and choose a norm $\lVert\cdotp\rVert$ on $V_{\chi}$.
 Let $d_{W}(\cdot,\cdot)$ be the word metric on $\Gamma$ (see \cite{LMR} for further details). Let $l_{W}(\gamma)=d_{W}(\gamma,e)$ be the length of 
 $\gamma\in\Gamma$ with respect to this metric.
 Then, if we put $C=\max\{\lVert\chi(\gamma_i)\rVert: \gamma_{i}\in L\cup L^{-1}\}$, we get for $c=\log C$,
 \begin{equation}
  \lVert\chi(\gamma)\rVert\leq e^{cl_{W}(\gamma)}.
 \end{equation}
 By \cite[Prop. 3.2]{LMR}, it follows that there exist positive constants $c_{1},c_{2}>0$ such that 
 \begin{equation}
  c_{1}d(x_{0},\gamma x_{0})\leq d_{W}(e,\gamma)\leq c_{2}d(x_{0}, \gamma x_{0}),
 \end{equation}
where $x_{0}:=eK$ is  the identity element of $\widetilde{X}$.
Then, (3.6) becomes by (3.7)
 \begin{equation*}
   \lVert\chi(\gamma)\rVert\leq C_{1}e^{c_{2}d(x_{0},\gamma x_{0})}.
 \end{equation*}
It follows that
\begin{equation}
 \lvert\tr\chi(\gamma)\rvert\leq\dim(V_{\chi})\lVert\chi(\gamma)\rVert\leq C_{3}e^{c_{2}d(x_{0},\gamma x_{0})}.
\end{equation}
Now by definition,
\begin{equation*}
 l(\gamma):=\min\{d(x,\gamma x):x\in\widetilde{X}\}.
\end{equation*}We choose a fundamental domain $F\subset\widetilde{X}$ for $\Gamma$ such that $x_{0}\in F$.
Given $\gamma\in\Gamma$, let $x_{1}$ be in $\widetilde{X}$, such that $l(\gamma)=d(x_{1},\gamma x_{1})$. Then, there exists a $\gamma_1\in\Gamma$
such that $ x_{1}\in \gamma_{1}F$. Let $x_{2}\in F$ such that $x_{1}=\gamma_{1}x_{2}$.
By compactness of the fundamental domain, $\diam(F)$ is finite. If we put $\delta:=\diam(F)$, then
\begin{equation}
 d(x_{0},x_{2})\leq \delta.
\end{equation}
We see that
\begin{align}
 d(x_{0},\gamma_{1}^{-1}\gamma\gamma_{1} x_{0})\notag&\leq d(x_{0},x_{2})+d(x_2,\gamma_{1}^{-1}\gamma\gamma_{1}x_{0})\\
&\leq \delta+d(x_2,\gamma_{1}^{-1}\gamma\gamma_{1}x_{0}).
\end{align}
In addition,
\begin{align}
d(x_2,\gamma_{1}^{-1}\gamma\gamma_{1}x_{0})\notag&\leq d(x_{2},\gamma_{1}^{-1}\gamma\gamma_{1}x_{2})+d(\gamma_{1}^{-1}\gamma\gamma_{1}x_{2},\gamma_{1}^{-1}\gamma\gamma_{1}x_{0})\\\notag
&\leq d(x_{2},\gamma_{1}^{-1}\gamma\gamma_{1}x_{2})+d(x_{0},x_{2})\\
&\leq d(x_{2},\gamma_{1}^{-1}\gamma\gamma_{1}x_{2})+\delta.
\end{align}
Hence, by (3.10) and (3.11) we get
\begin{equation*}
  d(x_{0},\gamma_{1}^{-1}\gamma\gamma_{1} x_{0})\leq 2\delta+ d(x_{2},\gamma_{1}^{-1}\gamma\gamma_{1}x_{2}).
\end{equation*}
Recall that $x_{1}=\gamma_{1}x_{2}$. Therefore, we have
\begin{align}
 d(x_{0},\gamma_{1}^{-1}\gamma\gamma_{1}x_{0})\notag&\leq 2\delta+d(\gamma_{1}^{-1}x_{1},\gamma_{1}^{-1}\gamma x_{1})\\
&\leq 2\delta+d(x_{1},\gamma x_{1}).
\end{align}
Using (3.8) and (3.12) we obtain the following inequalities.
\begin{align*}
 \lvert\tr(\chi(\gamma))\rvert&=\lvert\tr(\chi(\gamma_{1}^{-1}\gamma\gamma_{1}))\rvert\notag\\
 &\leq C_{3}e^{c_{2}d(x_{0},\gamma_{1}^{-1}\gamma\gamma_{1} x_{0})}\notag\\
 &\leq C_{3}e^{c_{2}(2\delta+d(x_{1},\gamma x_{1}))}\notag\\
 &=C_{4}e^{c_{2}d(x_{1},\gamma x_{1})}=C_{4}e^{c_{2}l(\gamma)}.
\end{align*} 
The assertion follows.
\end{proof}
We are ready now to prove the convergence of the Selberg and Ruelle zeta functions. 
\begin{prop}
Let $\chi\colon\Gamma\rightarrow \GL(V_{\chi})$ be a finite dimensional representation of $\Gamma$ and $\sigma\in \widehat{M}$. Then, there exists a constant $c>0$ such that 
 \begin{equation}
 Z(s;\sigma,\chi):=\prod_{\substack{[\gamma]\neq{e}\\ [\gamma]\prim}}\prod_{k=0}^{\infty}\det(1-(\chi(\gamma)\otimes\sigma(m_\gamma)\otimes S^k(\Ad(m_\gamma a_\gamma)_{\overline{\mathfrak{n}}}))e^{-(s+|\rho|)l(\gamma)}),
\end{equation}
converges absolutely and uniformly on compact subsets of the half-plane $\emph{\Re}(s)>c$.
\end{prop}
\begin{proof}
 We observe that 
 \begin{align}
\log Z(s;\sigma,\chi)\notag=&\sum_{\substack{[\gamma]\neq{e}\\ [\gamma]\prim}} \sum_{k=0}^{\infty}\tr\log(1-(\chi(\gamma)\otimes\sigma(m_\gamma)\otimes S^k(\Ad(m_\gamma a_\gamma)_{\overline{\mathfrak{n}}}))e^{-(s+|\rho|)l(\gamma)})\\\notag
=&-\sum_{\substack{[\gamma]\neq{e}\\ [\gamma]\prim}}\sum_{k=0}^{\infty}\sum_{j=1}^{\infty}\frac{\tr((\chi(\gamma)\otimes\sigma(m_\gamma)\otimes S^k(\Ad(m_\gamma a_\gamma)_{\overline{\mathfrak{n}}}))e^{-(s+|\rho|)l(\gamma)})^{j}}{j}\\\notag
=&-\sum_{[\gamma]\neq{e}}\sum_{k=0}^{\infty} \frac{1}{n_{\Gamma}(\gamma)}\tr(\chi(\gamma)\otimes\sigma(m_\gamma)\otimes S^k(\Ad(m_\gamma a_\gamma)_{\overline{\mathfrak{n}}}))e^{-(s+|\rho|)l(\gamma)}\\\notag
=&-\sum_{[\gamma]\neq{e}}\frac{1}{n_{\Gamma}(\gamma)}\tr(\chi(\gamma)\otimes\sigma(m_\gamma))\frac{e^{-(s+|\rho|)l(\gamma)}}{\det(\Id-\Ad(m_\gamma a_\gamma)_{\overline{\mathfrak{n}}})},\\
\end{align}
where in the last equation we made use of the identity
\begin{equation*}
 \sum_{k=0}^{\infty} S^k(\Ad(m_\gamma a_\gamma)_{\overline{\mathfrak{n}}})=\frac{1}{\det(\Id-\Ad(m_{\gamma}a_{\gamma})_{\overline{\mathfrak{n}}}}.
\end{equation*}
Initially, we observe that 
\begin{equation*}
 \lvert \tr\sigma(m_{\gamma})\rvert\leq \dim(\sigma), \quad\forall \sigma\in \widehat{M}.
\end{equation*}
We need an upper bound for the growth of the length spectrum $l(\gamma)$.
Using the normalization of the Haar measure on $G$ as in \cite[Proposition 7.6.4]{Wab} we see that there exists a
positive constant $C>0$ such that for every $R>0$ 
\begin{equation*}
 \Vol(B(x_{0},R))\leq Ce^{2|\rho| R},
\end{equation*}
where $\rho$ as in (2.3).
Since  $\Gamma$ is a cocompact lattice of $G$, there exists a positive constant $C'$ such that 
\begin{equation}
\sharp \{[\gamma]:l(\gamma)<R\}\leq\sharp\{ \gamma \in \Gamma: l(\gamma)\leq R\}\leq C'e^{2|\rho |R}.
\end{equation}
We need also an upper bound for the quantity
\begin{equation*}
 \frac{1}{\det(\Id-\Ad(m_\gamma a_\gamma)_{\overline{\mathfrak{n}}})}.
\end{equation*}
Since $
 \det(\Ad(a_{\gamma})_{\overline{\mathfrak{n}}})=\exp(-2|\rho| l(\gamma))
$
we can use the estimates (3.15) to see that we can consider a $[\gamma_{min}]$ among all the conjugacy classes of $\Gamma$, such that $l(\gamma_{min})$
is of minimum length.
Hence, there exists a positive constant $C''>0$ such that
\begin{equation*}
  \frac{1}{\det(\Id-\Ad(m_\gamma a_\gamma)_{\overline{\mathfrak{n}}})}<C''.
\end{equation*}
By Lemma 3.3, it follows that there exist positive constants $C,c_{1}>0$ such that
\begin{align*}
\sum_{[\gamma]\neq{e}}\frac{1}{n_{\Gamma}(\gamma)}\Big|\tr(\chi(\gamma)\otimes&\sigma(m_\gamma))\frac{e^{-(s+|\rho|)l(\gamma)}}
{\det(\Id-\Ad(m_\gamma a_\gamma)_{\overline{\mathfrak{n}}})}\Big|\\
&\leq C \sum_{[\gamma]\neq{e}}e^{(c_{1}-\text{\Re}(s))l(\gamma)}\\
&= C\sum_{k=0}^{\infty}\sum_{\substack{[\gamma]\neq {e}\\ k\leq l(\gamma)\leq k+1}}e^{(c_{1}-\text{\Re}(s)) l(\gamma)}\\
&\leq C\sum_{k=0}^{\infty}\mathcal{N}(k+1) e^{(c_{1}-\text{\Re}(s))k},\\
\end{align*}
where 
\begin{equation*}
 \mathcal{N}(R):=\sharp\{[\gamma] \in \Gamma: l(\gamma)\leq R\},\quad R\geq 0.
\end{equation*}
By (3.15), we have 
\begin{equation*}
 \sum_{k=0}^{\infty}\mathcal{N}(k+1) e^{(c_{1}-\text{\Re}(s))k}\leq C' \sum_{k=0}^{\infty} e^{(2|\rho|+c_{1}-\text{\Re}(s))k}.
\end{equation*}
Hence, there exists a positive constant $c>0$ such that for $s\in\C$ with $\text{\Re}(s)>c$,
\begin{equation*}
\sum_{[\gamma]\neq{e}}\frac{1}{n_{\Gamma}(\gamma)}\Big|\tr(\chi(\gamma)\otimes\sigma(m_\gamma))\frac{e^{-(s+|\rho|)l(\gamma)}}
{\det(\Id-\Ad(m_\gamma a_\gamma)_{\overline{\mathfrak{n}}})}\Big|\\
 <\infty.
\end{equation*}
The assertion follows form (3.14).
\end{proof}

A similar approach will be used to establish the convergence of the Ruelle zeta function.
\begin{prop}
Let $\chi\colon\Gamma\rightarrow \GL(V_{\chi})$ be a finite dimensional representation of $\Gamma$ and $\sigma\in\widehat{M}$. Then, there exists a constant $r>0$ such that 
 \begin{equation}
 R(s;\sigma,\chi):=\prod_{\substack{[\gamma]\neq{e}\\ [\gamma]\prim}}\det\big(\Id-\chi(\gamma)\otimes\sigma(m_{\gamma})e^{-sl(\gamma)}\big)^{(-1)^{d-1}}.
\end{equation}
converges absolutely and uniformly on compact subsets of the half-plane \emph{Re}$(s)>r$.
\end{prop}
\begin{proof}
 We observe that 
  \begin{align}
\log R(s;\sigma,\chi)=\notag&(-1)^{d-1}\sum_{\substack{[\gamma]\neq{e}\\ [\gamma]\prim }} \tr\log(1-\chi(\gamma)\otimes \sigma(m_{\gamma})e^{-sl(\gamma)})\\\notag
&=(-1)^{d}\sum_{\substack{[\gamma]\neq{e}\\  [\gamma]\prim }}\sum_{j=1}^{\infty}\frac{\tr((\chi(\gamma)\otimes \sigma(m_{\gamma})e^{-sl(\gamma)})^{j})}{j}\\
&=(-1)^{d}\sum_{[\gamma]\neq{e}}\frac{1}{n_{\Gamma}(\gamma)} \tr(\chi(\gamma)\otimes \sigma(m_{\gamma}))e^{-sl(\gamma)}.
\end{align}
By Lemma 3.3, it follows that there exist positive constants $C,c_{1}>0$ such that
\begin{align*}
\sum_{[\gamma]\neq{e}}\frac{1}{n_{\Gamma}(\gamma)} \Big|\tr(\chi(\gamma)\otimes &\sigma(m_{\gamma}))e^{-sl(\gamma)}\Big |\\
&\leq C\sum_{[\gamma]\neq{e}} e^{(c_{1}-\text{\Re}(s))l(\gamma)}\\
&=C\sum_{k=0}^{\infty}\sum_{\substack{[\gamma]\neq e\\ k\leq l(\gamma)\leq k+1}} \mathcal{N}(k+1)e^{(c_{1}-\text{\Re}(s))l(\gamma)}\\
&\leq C\sum_{k=0}^{\infty}\mathcal{N}(k+1)e^{(c_{1}-\text{\Re}(s))k}.\\
\end{align*}
By (3.15), we have
\begin{equation*}
 \sum_{k=0}^{\infty}\mathcal{N}(k+1)e^{(c_{1}-\text{\Re}(s))k}\leq C'\sum_{k=0}^{\infty}e^{(2|\rho|+c_{1}-\text{\Re}(s))k}.
\end{equation*}
Hence, there exists a positive constant $r>0$ such that for $s\in\C$ with $\text{\Re}(s)>r$,
\begin{equation}
\sum_{[\gamma]\neq{e}}\frac{1}{n_{\Gamma}(\gamma)} \Big|\tr(\chi(\gamma)\otimes \sigma(m_{\gamma}))e^{-sl(\gamma)}\Big |<\infty.
\end{equation}
The assertion follows form (3.17).
\end{proof}
\begin{lem}We set
\begin{equation}
 L_{sym}(\gamma;\sigma):=\frac{\tr(\chi(\gamma)\otimes\sigma(m_{\gamma}))e^{-|\rho|l(\gamma)}}
 {\det{(\Id-\Ad(m_{\gamma}a_{\gamma})_{\overline{n}})}}.
\end{equation}
Then, the logarithmic derivative of the Selberg zeta function $Z(s;\sigma,\chi)$ is given by
\begin{equation}\label{log der selberg}
 L(s):=\frac{d}{ds}\log(Z(s;\sigma,\chi))=\sum_{[\gamma]\neq{e}}\frac{l(\gamma)}{n_{\Gamma}(\gamma)}L_{sym}(\gamma;\sigma)
e^{-sl(\gamma)}.
\end{equation}
\end{lem}
\begin{proof}
We see by equation (3.14)
\begin{align*}
\frac{d}{ds}\log(Z(s;\sigma,\chi))&=\sum_{[\gamma]\neq{e}}\frac{l(\gamma)}{n_{\Gamma}(\gamma)}\tr(\chi(\gamma)\otimes\sigma(m_\gamma))
\frac{e^{-sl(\gamma)}e^{-|\rho|l(\gamma)}}{\det(1-\Ad(m_\gamma a_\gamma)_{\overline{\mathfrak{n}}})}\\ 
 &=\sum_{[\gamma]\neq e}\frac{l(\gamma)}{n_{\Gamma}(\gamma)}L_{sym}(\gamma;\sigma)
e^{-sl(\gamma)}.
\end{align*}
\end{proof}

\begin{center}
\section{\textmd{The twisted Bochner-Laplace operator}}
\end{center}

Let $E\rightarrow X$ be a complex vector bundle with covariant derivative $\nabla$. We define the second covariant derivative $\nabla^2$ by
\begin{equation*}
\nabla^2_{V,W}:=\nabla_{V}\nabla_{W}-\nabla_{\nabla^{LC}_{V}W}, 
\end{equation*}
where $V,W\in C^{\infty}(X,TX)$ and $\nabla^{LC}$ denotes the Levi-Civita connection on $TX$.
We define the connection Laplacian $\Delta_{E}$ to be the negative of the trace of the second covariant derivative, i.e., 
\begin{equation*}
 \Delta_{E}:=-\Tr\nabla^2.
\end{equation*}
By \cite[p.154]{LM}, the connection Laplacian is equal to the Bocner-Laplace operator, i.e.
\begin{equation*}
 \Delta_{E}=\nabla^{*}\nabla.
\end{equation*}
In terms of a local orthonormal frame field $(e_{1},\ldots,e_{d})$ of $T_{x}X$, for $x\in X$,
the connection Laplacian is given by
\begin{equation*}
\Delta_{E}=-\sum_{j=1}^{d}\nabla^ 2{_{e_j,e_j}}.
\end{equation*}
$\Delta_{E}\colon C^{\infty}(X,E)\circlearrowright$ is a second order differential operator.
Let $h$ be a metric in $E$. Then, $\Delta_{E}$ acts in $L^{2}(X,E)$ with domain $C^{\infty}(X,E)$. 
Since the principal symbol of $\Delta_{E}$ is computed to be $\sigma_{\Delta_{E}}(x,\xi)
 =\lVert \xi \rVert^ {2}_{x} \Id_{E_{x}}$, we can conclude that $\Delta_{E}$ is an elliptic operator
and hence it has nice spectral properties. Namely, its spectrum is discrete and contained in a translate of a positive cone $C\subset \C$ such that $\R^{+}\subset C$
(\cite[Theorem 8.4 and Theorem 9.3]{Sh}).
Furthermore, if the metric is compatible with the connection $\nabla$, $\Delta_{E}$ is formally self-adjoint. 

Let $\chi\colon\Gamma\rightarrow \GL(V_{\chi})$ be an arbitrary representation $\chi\colon\Gamma\rightarrow \GL(V_{\chi})$ of $\Gamma$. 
Let $E_{\chi}\rightarrow X$ be the associated flat vector bundle 
over $X$, equipped with a flat connection $\nabla^{E_{\chi}}$.

We specialize now to the twisted case  $E=E_{0}\otimes E_{\chi}$, where $E_{0}\rightarrow X$ is a complex vector bundle equipped with a connection $\nabla^{E_{0}}$ and 
a metric, which is compatible with this connection.
Let $\nabla^{E}=\nabla^{E_{0}\otimes E_{\chi}}$ the product-connection, defined by
\begin{equation*}
 \nabla^{E_{0}\otimes E_{\chi}}:=\nabla^{E_{0}}\otimes 1+1\otimes\nabla^{E_{\chi}}.
\end{equation*}
 We define the operator $\Delta_{E_{0},\chi}^\sharp$ by 
\begin{equation}
 \Delta_{E_{0},\chi}^{\sharp}=-\Tr\big((\nabla^{E_{0}\otimes E_{\chi}})^2\big).
\end{equation}
We choose a hermitian metric in $E_{\chi}$.
Then, $\Delta_{E_{0},\chi}^{\sharp}$ acts on $L^{2}(X,E_{0}\otimes E_{\chi})$. However, it is not a formally self-adjoint operator in general.
We want to describe this operator locally. Following the analysis in \cite{M1}, we consider an open subset $U$ of $X$ such that $E_{\chi}\lvert_{U}$ is trivial.
Then, $E_{0}\otimes E_{\chi}\lvert_{U}$ is isomorphic to the direct sum of $m$-copies of $E_{0}\lvert_{U}$, i.e., 
\begin{equation*}
 (E_{0}\otimes E_{\chi})\lvert_{U}\cong\oplus_{i=1}^{m}E_{0}\lvert_{U},
\end{equation*}
where $m:=\rank(E_\chi)=\dim V_{\chi}$.
Let $(e_{i}),i=1,\cdots,m$ be a basis of flat sections of $E_{\chi}\lvert_{U}$.
Then, each $\phi \in C^{\infty}(U,(E_{0}\otimes E_{\chi})\lvert_{U})$ can be written as
\begin{equation*}
\phi=\sum_{i=1}^{m}\phi_{i} \otimes e_{i},
\end{equation*}
where $\phi_{i}\in C^{\infty}(U, E_{0}\lvert_{U}), i=1,\ldots,m$.
The product connection is given by
\begin{equation*}
  \nabla_{Y}^{E_{0}\otimes E_{\chi}}\phi=\sum_{i=1}^{m}(\nabla_{Y}^{E_{0}})(\phi_{i})\otimes e_{i},
\end{equation*}
where  $Y\in C^{\infty}(X,TX)$.\\
By (4.1) we obtain the twisted Bochner-Laplace operator acting on $C^{\infty}(X,E_{0}\otimes E_{\chi})$, defined by
\begin{equation}
 \Delta_{E_{0},\chi}^{\sharp}\phi=\sum_{i=1}^{m}(\Delta_{E_{0}}\phi_{i})\otimes e_{i},
\end{equation}
where $\Delta_{E_{0}}$ denotes the Bochner-Laplace operator $\Delta^{E_{0}}=(\nabla^{E_{0}})^*\nabla^{E_{0}}$
associated to the connection $\nabla^{E_{0}}$.
Let now $\widetilde{E}_{0}, \widetilde{E}_{\chi}$ be the pullbacks to $\widetilde{X}$ of $E_{0},E_{\chi}$, respectively. 
Then,
\begin{equation*}
 \widetilde{E}_{\chi}\cong \widetilde{X}\times V_{\chi},
\end{equation*}
and
\begin{equation}
 C(\widetilde{X}, \widetilde{E}_{0}\otimes \widetilde{E}_{\chi})\cong  C(\widetilde{X}, \widetilde{E}_{0})\otimes V_{\chi}.
\end{equation}
With respect to the isomorphism (4.3), it follows from (4.2) that the lift of $\Delta_{E_{0},\chi}^{\sharp}$ to $\widetilde{X}$ takes the form
\begin{equation}
 \widetilde{\Delta}^{\sharp}_{E_{0},\chi}=\widetilde{\Delta}_{E_{0}}\otimes \Id_{V_{\chi}},
\end{equation}
where $\widetilde{\Delta}_{E_{0}}$ is the lift of $\Delta_{E_{0}}$ to $\widetilde{X}$.
By (4.2), $\Delta_{E_{0},\chi}^{\sharp}$ has principal symbol 
\begin{equation*}
 \sigma_{\Delta_{E_{0},\chi}^\sharp}(x,\xi)=\lVert \xi \rVert^ {2}_{x} \Id_{({E_{0}\otimes E_{\chi})_{x}}}, \quad x\in X, \xi\in\T_{x}^{*}X, \xi\neq0.
\end{equation*}
Hence, since the principal symbol is self-adjoint with respect to 
the fiber metrics on $E\chi$ and $E_{0}$,
it has nice spectral properties, i.e.,
its spectrum is discrete and contained in a translate of a positive cone $C\subset \C$ such that $\R^{+}\subset C$
(\cite[Theorem 8.4 and Theorem 9.3]{Sh}).
We include here some definitions, which are needed to study the spectrum of the
twisted Laplace-Bochner operator.
For further details see \cite[p.203-206]{BK2}.      
\begin{defi}
 A spectral cut is a ray
\begin{equation*}
 R_{\theta}:=\{\rho e^ {i\theta}: \rho\in[0,\infty]\},
\end{equation*}
where $\theta\in[0,2\pi)$.
\end{defi}
\begin{defi}
 The angle $\theta$ is a principal angle for the elliptic operator $\Delta_{E_{0},\chi}^{\sharp}$ if 
\begin{equation*}
 \spec(\sigma_{\Delta}(x,\xi))\cap R_{\theta}=\emptyset,\quad \forall x\in X,\forall\xi\in T_{x}^{*}X,\xi\neq 0.
\end{equation*}
\end{defi}
\begin{defi}
 We define the solid angle $L_{I}$ associated with a closed interval $I$ of $\R$ by
\begin{equation*}
 L_{I}:=\{\rho e^ {i\theta}: \rho\in(0,\infty), \theta\in I \}.
\end{equation*}
\end{defi}
\begin{defi}
The angle $\theta$ is an Agmon angle for an elliptic operator $\Delta_{E_{0},\chi}^{\sharp}$, if it is a principal angle for $D$ 
and there exists $\varepsilon>0$ such that 
\begin{equation*}
 \spec(\Delta)\cap L_{[\theta-\varepsilon,\theta+\varepsilon]}=\emptyset.
\end{equation*}
\end{defi}
\begin{lem}
Let $\varepsilon\in(0,\frac{\pi}{2})$ be an angle such that the principal symbol $\sigma_{{\Delta}^{\sharp}_{E_{0},\chi}}(x,\xi)$ of ${\Delta}^{\sharp}_{E_{0},\chi}$,
for $\xi\in T_{x}^{*}X,\xi\neq 0$ does not take values in $ L_{[-\varepsilon,\varepsilon]}$.
Then, the spectrum $\spec({\Delta}^{\sharp}_{E_{0},\chi})$ of the operator ${\Delta}^{\sharp}_{E_{0},\chi}$ 
is discrete and for every $\varepsilon\in(0,\frac{\pi}{2})$ there exist $R>0$ such that $\spec({\Delta}^{\sharp}_{E_{0},\chi})$
is contained in the set $B(0,R)\cup L_{[-\varepsilon,\varepsilon]}\subset \C$.
\end{lem}
\begin{proof}
 The discreteness of the spectrum follows from \cite[Theorem 8.4]{Sh}. For the second statement see \cite[Theorem 9.3]{Sh}.
\end{proof}
Let $\lambda_{k}$ be an eigenvalue of ${\Delta}^{\sharp}_{E_{0},\chi}$ and 
$V_{\lambda_{k}}$ be the corresponding eigenspace.
This is a finite dimensional subspace of $C^{\infty}(X,E)$ invariant under $D$.
For every $k\in\N$, there exist $N_{k}\in\N$ such that 
\begin{align*}
 &({\Delta}^{\sharp}_{E_{0},\chi}-\lambda_{k}\Id)^{N_{k}}V_{\lambda_{k}}=0\\
&\lim_{k\rightarrow \infty}\lvert \lambda_{k}\rvert=\infty.
\end{align*}
By \cite{Mk}, the space $L^{2}(X,E)$ can be decomposed as
\begin{equation*}
 L^{2}(X,E)=\overline{\bigoplus_{k\geq 1}V_{\lambda_{k}}}.
\end{equation*}
This is the generalization of the eigenspace decomposition of a self-adjoint operator. 
We note here that in general the above decomposition is not a sum of mutually orthogonal subspaces (see \cite[p. 7]{M1}).
\begin{defi}
We call algebraic multiplicity $m(\lambda_{k})$ of the eigenvalue $\lambda_{k}$ the dimension of the corresponding
eigenspace $V_{\lambda_{k}}$.
\end{defi}
We want to define the heat operator $e^{-t\Delta_{E_{0},\chi}^\sharp}$ associated to the twisted Bochner-Laplace operator.
Let $\theta$ be an Agmon angle for the operator $\Delta_{E_{0},\chi}^\sharp$. Then, by definition of the Agmon angle and Lemma 4.5, there exists $\varepsilon>0$
such that 
\begin{equation*}
 \spec(\Delta_{E_{0},\chi}^\sharp)\cap L_{[\theta-\varepsilon,\theta+\varepsilon]}=\emptyset.
\end{equation*}
 Since $\Delta_{E_{0},\chi}^\sharp$ has discrete spectrum, there exists also an $r_{0}>0$ such that 
\begin{equation*}
\spec(\Delta_{E_{0},\chi}^\sharp)\cap\{z\in\C:\lvert z+1\rvert\leq 2r_{0}\}=\emptyset.
\end{equation*}
 We define a contour $\Gamma_{\theta,r_{0}}$ as follows.
\begin{equation*}
 \Gamma_{\theta,r _{0}}=\Gamma_{1}\cup\ \Gamma_{2}\cup\Gamma_{3},
\end{equation*}
where $\Gamma_{1}=\{-1+re^ {i\theta}\colon\infty>r\geq r_{0}\}$, $\Gamma_{2}=\{-1+r_{0}e^ {ia}\colon\theta\leq a\leq \theta+2\pi\}$,
$\Gamma_{3}=\{-1+re^{i(\theta+2\pi)}\colon r_{0}\leq r< \infty\}$.
On $\Gamma_{1}$, $r$ runs from $\infty$ to $r_{0}$, $\Gamma_{2}$ is oriented counterclockwise, and on $\Gamma_{3}$,
$r$ runs from $r_{0}$ to $\infty$.
We put
\begin{equation}
e^{-t\Delta_{E_{0},\chi}^{\sharp}}=\frac{i}{2\pi}\int_{\Gamma_{\theta,r_{0}}}e^{-t\lambda}(\Delta_{E_{0},\chi}^{\sharp}-\lambda\Id)^{-1} d\lambda.
\end{equation}
By \cite[Corollary 9.2]{Sh} and the fact that $\lvert e^{-t\lambda}\rvert\leq e^{-t\text{\Re}(\lambda)}$, the integral in equation (4.5)
is well defined.
\begin{center}
\section{\textnormal{Trace Formulas}}
\end{center}
We want now to define a more special case of a twisted Bochner-Laplace operator.
Namely, we conisder the operator $\Delta^{\sharp}_{\tau,\chi}$
acting on $C^{\infty}(X,E_{\tau}\otimes E_{\chi})$, where $E_{\tau}$
is the locally associated homogenous vector bundle
associated with a complex finite dimensional unitary representaion $(\tau,V_{\tau})$ of $K$.

The keypoint is that, when we consider the lift of the twisted Bochner-Laplace operator to the universal covering, this operator acts as the identity operator on the space
of the smooth sections of the flat vector bundle $E_{\chi}$.
Recall that by formula (4.4), we get
\begin{equation*}
 \widetilde{\Delta}^{\sharp}_{\tau,\chi}=\widetilde{\Delta}_{\tau}\otimes \Id_{V_{\chi}},
\end{equation*}
where $\widetilde{\Delta}_\tau$ is the lift to $\widetilde{X}$ of the Bochner-Laplace operator $\Delta_{\tau}$, associated with the
representation $\tau$ of $K$.

We give here an explicit description of the operator $\Delta_{\tau}$.
We regard the Lie group $G$ as principal $K$-fiber bundle over $\widetilde{X}$. Let $\pi\colon G\rightarrow G/K$ be the canonical projection.
Then, since $\mathfrak{p}$ is invariant under the adjoint action $\Ad(k), k\in K$, the assignment
\begin{equation*}
 T_{g}^{hor}:={\frac{d}{dt}\bigg|_{t=0}g\exp(tX), \quad X\in \mathfrak{p}}
\end{equation*}
defines a horizontal distribution on $G$ (\cite[Chapter \MakeUppercase{\romannumeral 3}]{KNo}). This is the canonical connection in the principal bundle $G$.
Let $\tau:K\rightarrow \GL(V_\tau)$ be a complex finite dimensional unitary representation of $K$ on a vector space $V_\tau$, equipped with an inner product $\langle\cdotp,\cdotp\rangle_{\tau}$.
Let $\widetilde{E}_{\tau}$ be the homogenous vector bundle associated with $(\tau,V_{\tau})$, defined by
\begin{equation*}
\widetilde{E}_{\tau}:=G\times_{\tau} V_{\tau}\rightarrow\widetilde{X},
\end{equation*}
where $K$ acts on $(G,V_{\tau})$ on the right  by
\begin{equation*}
 (g,v)k=(gk,\tau^{-1}(k)v), \quad g\in G, k\in K, v\in V_{\tau}.
\end{equation*}
The  inner product $\langle \cdotp, \cdotp \rangle_{\tau}$ on the vector space $V_{\tau}$
induces a $G$-invariant metric $h^{E_\tau}$ on $\widetilde{E}_{\tau}$.
We denote by $C^{\infty}(\widetilde{X},\widetilde{E}_{\tau})$ the space of the smooth sections
of the vector bundle $\widetilde{E}_{\tau}$.
We define the space
\begin{equation}
 C^{\infty}(G;\tau)=\{f:G\rightarrow V_{\tau}\colon f\in C^{\infty}, f(gk)=\tau(k)^{-1}f(g), \forall g\in G, \forall k\in K\}.
\end{equation} 
Similarly, we denote by $C^{\infty}_{c}(G;\tau)$ the subspace of $C^{\infty}(G;\tau)$ of compactly supported functions and 
$L^{2}(G;\tau)$ the completion of $C^{\infty}_{c}(G;\tau)$ with respect to the inner product
\begin{equation*}
 \langle f,h\rangle=\int_{G/K}\langle f(g),h(g)\rangle_{\tau} d\dot{g}.
\end{equation*}
Let $A\colon C^{\infty}(\widetilde{X},\widetilde{E}_{\tau})\rightarrow C^{\infty}(G;\tau)$ be the operator, defined by
\begin{equation*}
 Af(g)=g^{-1}f(gK).
\end{equation*}
Then, the canonical connection on $\widetilde{E}_{\tau}$ is given by
\begin{align*}
 A(\nabla_{d\pi(g)X}^{\tau}f)(g)&=\frac{d}{dt}\bigg|_{t=0}Af(g\exp(tX))\\
 &=\frac{d}{dt}\bigg|_{t=0}(g\exp(tX))^{-1}f(g\exp(tX)K),
\end{align*}
where $g\in G, X\in \mathfrak{p}$, and $f\in C^{\infty}(\widetilde{X},\widetilde{E}_{\tau})$.
By \cite[p. 4]{Mi1}, $A$ induces a canonical isomorphism
\begin{equation}
C^{\infty}(\widetilde{X},\widetilde{E}_{\tau})\cong C^{\infty}(G;\tau).
\end{equation}
Similarly, there exist the following isomorphisms 
\begin{align}
&C_{c}^{\infty}(\widetilde{X},\widetilde{E}_{\tau})\cong C^{\infty}_{c}(G;\tau);\\\notag
&L^{2}(\widetilde{X},\widetilde{E}_{\tau})\cong L^{2}(G;\tau).
\end{align}
We consider the Bochner-Laplace operator associated with $\widetilde{\nabla}^{\tau}$,
\begin{equation*}
 \widetilde{\Delta}_{\tau}=(\widetilde{\nabla}^{\tau})^{*}\widetilde{\nabla}^{\tau}:C_{c}^{\infty}(\widetilde{X},\widetilde{E}_{\tau})\rightarrow L^{2}(\widetilde{X},\widetilde{E}_{\tau}).
\end{equation*}
Let now $\Omega \in Z(\mathfrak{g_{\C}})$ be the Casimir element of $G$. We assume that $\tau$ is irreducible.
Let $\Omega|_{K}\in Z(\mathfrak{k})$ be the Casimir element of $K$ and $\lambda_{\tau}$ the associated Casimir eigenvalue,
where $Z(\mathfrak{k})$ denotes the center of $\mathfrak{k}$.
Then, with respect to the isomorphism (5.2), the Bochner-Laplace operator acting on $C^{\infty}_{c}(G;\tau)$ is given by
\begin{equation}
 \widetilde{\Delta}_{\tau}=-R(\Omega)+\lambda_{\tau}\Id.
\end{equation}
This is proved in \cite[Proposition 1.1]{Mi1}.
The operator $\widetilde{\Delta}_{\tau}$ is an elliptic formally self-adjoint differential operator of second order.
By \cite{Ch}, it is an essentially self-adjoint operator. Its self-adjoint extension will be also denoted by $\widetilde{\Delta}_{\tau}$.\\
We consider the corresponding heat semi-group $e^{-t\widetilde{\Delta}_{\tau}}$ acting on the space $L^2(\widetilde{X}, \widetilde{E}_{\tau})$.
\begin{equation*}
 e^{-t\widetilde{\Delta}_{\tau}}\colon L^2(\widetilde{X}, \widetilde{E}_{\tau})\rightarrow L^2(\widetilde{X}, \widetilde{E}_{\tau}).
\end{equation*}
By \cite[p.467]{CY}, $e^{-t\widetilde{\Delta}_{\tau}},t>0$ is an infinitely smoothing operator with a $C^{\infty}$-kernel, 
i.e.  there exists a smooth function $k_{t}^{\tau}\colon G\times G\rightarrow \End(V_{\tau})$ such that 
\begin{enumerate}
\item it is symmetric in the $G$-variables and for each $g\in G$, the map
$g'\mapsto k_{t}^{\tau}(g,g')$
belongs to $L^2(\widetilde{X}, \widetilde{E}_{\tau})$;
\item it satisfies the covariance property,
\begin{equation*}
 k_{t}^{\tau}(gk,g'k')=\tau^{-1}(k)k_{t}^{\tau}(g,g') \tau(k'), \quad 
 \forall g,g'\in G, k,k'\in K;
\end{equation*}
\item for  $f\in L^{2}(\widetilde{X}, \widetilde{E}_{\tau})$,
\begin{equation}
 e^{-t\widetilde{\Delta}_{\tau}}f(g)=\int_{G}k_{t}^{\tau}(g,g')f(g')dg'.
\end{equation}
\end{enumerate}
The Casimir element is invariant under the action of $G$. Hence, $\widetilde{\Delta}_{\tau}$ is $G$-invariant, and $e^{-t\widetilde{\Delta}_{\tau}}$ is an integral 
operator which commutes with the right regular representation of $G$ in $L^{2}(\widetilde{X},\widetilde{E}_{\tau})$.
Then there exists a function $H_{t}^{\tau}\colon G\rightarrow\End(V_{\tau})$, such that 
\begin{enumerate}
\item  $H_{t}^{\tau}(g^{-1}g')=k_{t}^{\tau}(g,g'),\quad \forall g,g'\in G$;
\item  it satisfies the covariance property
\begin{equation}
 H_{t}^{\tau}(kgk')=\tau^{-1}(k)H_{t}^{\tau}(g) \tau(k'),\quad \forall g\in G, \forall k,k'\in K;
\end{equation}
\item for $f\in L^{2}(\widetilde{X}, \widetilde{E}_{\tau})$,
\begin{equation}
 e^{-t\widetilde{\Delta}_{\tau}}f(g)=\int_{G}H_{t}^{\tau}(g^{-1}g')f(g')dg'.
\end{equation}
\end{enumerate}
We denote by $(\mathcal{C}^{q}(G)\otimes \End(V_{\tau}))^{K\times K}$ the Harish-Chandra $L^{q}$-Schwartz space of functions on $G$
with values in $\End(V_{\tau})$, defined as in \cite[p.161-162]{BM}, such that the covariance property (5.6) is satisfied.
\begin{thm}
Let $t>0$. Then, for every $q>0$
\begin{equation*}
 H_{t}^{\tau}\in (\mathcal{C}^{q}(G)\otimes \End(V_{\tau}))^{K\times K}.
\end{equation*}
\end{thm}
\begin{proof}
 This is proved in \cite[Proposition 2.4]{BM}.
\end{proof}
In \cite[p.161]{BM}, it is proved that 
\begin{equation*}
 e^{{-t\widetilde{\Delta}_{\tau}}}=R_{\Gamma}(H_{t}^{\tau}),
\end{equation*}
where $R_{\Gamma}(H_{t}^{\tau})$ denotes the bounded trace class operator,
induced by the right regular representaion of $G$, acting on $C^{\infty}(G;\tau)$.
It is decribed by the formula
\begin{equation*}
  e^{-t\widetilde{\Delta}_{\tau}}f(g)=\int_{G}H_{t}^{\tau}(g^{-1}g')f(g')dg'.
\end{equation*}
More generally, we consider a unitary admissible representation $\pi$ of $G$ in a Hilbert space $\mathcal{H}_{\pi}$. We set
\begin{equation*}
 \widetilde{\pi}(H_{t}^{\tau})=\int_{G}\pi(g)\otimes H_{t}^{\tau}(g)dg.
\end{equation*}
This defines a bounded trace class operator on $\mathcal{H}_{\pi}\otimes V_{\tau}$.
By \cite[p.160-161]{BM}, relative to the splitting
\begin{equation*}
 \mathcal{H}_{\pi}\otimes V_{\tau}=(\mathcal{H}_{\pi}\otimes V_{\tau})^K\oplus [(\mathcal{H}_{\pi}\otimes V_{\tau})^K]^{\perp},
\end{equation*}
$ \widetilde{\pi}(H_{t}^{\tau})$ has the form
\begin{equation}
  \widetilde{\pi}(H_{t}^{\tau})= \begin{pmatrix}
\pi(H_{t}^{\tau}) & 0 \\
 0 & 0
\end{pmatrix},
\end{equation}
with $\pi(H_{t}^{\tau})$ acting on $(\mathcal{H}_{\pi}\otimes V_{\tau})^K$.
Then, it follows that 
\begin{equation}
 e^{-t(-\pi(\Omega)+\lambda_{\tau})}\Id=\pi(H_{t}^{\tau}),
\end{equation}
where $\Id$ denotes the identity on the space $(\mathcal{H}_{\pi}\otimes V_{\tau})^K$ (\cite[Corollary 2.2]{BM}).
We let 
\begin{equation*}
 h_{t}^{\tau}(g):=\tr H_{t}^{\tau}(g).
\end{equation*}
We consider orthonormal bases $(\xi_{n}), n\in \N, (e_{j}), j=1,\cdots,k$ of the vector spaces $\mathcal{H}_{\pi}, V_{\tau}$, respectively, where $k:=\dim(V_{\tau})$.
By (5.8),
\begin{equation}
 \Tr(\pi(H_{t}^{\tau}))=\Tr(\widetilde{\pi}(H_{t}^{\tau})).
\end{equation}
We have
\begin{align}
 \Tr(\widetilde{\pi}(H_{t}^{\tau}))=\notag&\sum_{n}\sum_{j}\langle \widetilde{\pi}(H_{t}^{\tau})(\xi_n\otimes e_j),(\xi_n\otimes e_j)\rangle\\\notag
 &=\sum_{n}\sum_{j}\int_{G}\langle \pi(g)\xi_n,\xi_n\rangle\langle H_{t}^{\tau}(g)e_j,e_j\rangle dg\\\notag
 &=\sum_{n}\int_{G}\langle\pi(g)\xi_n,\xi_n\rangle h_{t}^{\tau}(g)dg\\\notag
 &=\sum_{n}\langle\pi(h_{t}^{\tau})\xi_n,\xi_n\rangle\\
 &=\Tr\pi(h_{t}^{\tau}).
 \end{align}
Hence, if we combine equations (5.9), (5.10) and (5.11), we get
\begin{equation}
 \Tr\pi(h_{t}^{\tau})=e^{-t(-\pi(\Omega)+\lambda_{\tau})}\dim(\mathcal{H}_{\pi}\otimes V_{\tau})^K.
\end{equation}
Now we want to specify the unitary representation $\pi$ of $G$. We consider the unitary principal series representation $\pi_{\sigma,\lambda}$,
defined in section 2.
Our goal is to compute the Fourier transform of $h_{t}^{\tau}$,
\begin{equation*}
\Theta_{\sigma,\lambda}(h_{t}^{\tau})=\Tr\pi_{\sigma,\lambda}(h_{t}^{\tau}).
\end{equation*}
\begin{prop}
 For $\sigma \in \widehat{M}$ and $\lambda \in \R$, let $\Theta_{\sigma, \lambda}$ be the global character of $\pi_{\sigma,\lambda}$.
 Let $\tau\in \widehat{K}$. Then,
 \begin{equation}
  \Theta_{\sigma,\lambda}(h_{t}^{\tau})=e^{-t(-\pi_{\sigma,\lambda}(\Omega)+\lambda_{\tau})}.
  \end{equation}
\end{prop}
\begin{proof}
 We have
 \begin{equation*}
  \Theta_{\sigma,\lambda}(h_{t}^{\tau})=e^{-t(-\pi_{\sigma,\lambda}(\Omega)+\lambda_{\tau})}\dim(\mathcal{H}_{\pi_{\sigma,\lambda}}\otimes V_{\tau})^K
=e^{-t(-\pi_{\sigma,\lambda}(\Omega)+\lambda_{\tau})}[\pi_{\sigma,\lambda}:\check{\tau}],
  \end{equation*}
where $\check{\tau}$ denotes the contragredient representation of $\tau$.
We have
\begin{align*}
\Theta_{\sigma,\lambda}(h_{t}^{\tau})=&e^{-t(-\pi_{\sigma,\lambda}(\Omega)+\lambda_{\tau})}[\pi_{\sigma,\lambda}:\check{\tau}]\\
&e^{-t(-\pi_{\sigma,\lambda}(\Omega)+\lambda_{\tau})}[\pi_{\sigma,\lambda}:{\tau}]=e^{-t(-\pi_{\sigma,\lambda}(\Omega)+\lambda_{\tau})}[\tau\mid_{M}:\sigma].
\end{align*}
In the last line in the equation above we use the Frobenius reciprocity principle, 
which is described for compact Lie groups in (\cite[Theorem 1.14]{Knapp}).
By (\cite[p.208]{Knapp}), one has
\begin{equation*}
 [\pi_{\sigma,\lambda}|_{K}:{\tau}]=\sum_{\omega\in (M\cap K)^{\widehat{}}}n_{\omega}
 [\tau\mid_{M\cup K}:\omega],
\end{equation*}
where $n_{\omega}$ are positive integers.
But, in our case $M\subset K$ and therefore $M\cap K=M$. \\
Hence,
\begin{equation*}
[\pi_{\sigma,\lambda}|_{K}:{\tau}]=[\tau\mid_{M}:\sigma].
\end{equation*}
By \cite[Theorem 8.1.3]{GW}, $K$ is multiplicity free in $G$, i.e. $[\pi_{\sigma,\lambda}:{\tau}]\leq 1$.
The assertion follows.
\end{proof}
We pass now to $X=\Gamma\backslash \widetilde{X}$. We consider the
locally homogeneous vector bundle
\begin{equation*}
 E_{\tau}:=\Gamma\backslash\widetilde{E}_{\tau}\rightarrow X.
\end{equation*}
Let  $E_{\chi}$ be the flat vector bundle over $X$.
We want to derive a trace formula for the heat operator $e^{-t\Delta_{\tau,\chi}^{\sharp}}$.
By Lemma 2.4 and Proposition 2.5 in \cite{M1}, $e^{-t\Delta_{\tau,\chi}^{\sharp}}$
is an integral operator with smooth kernel and of trace class.

We can apply the Lidskii's theorem,
which gives a general expression for the trace of a trace class (not necessarily self-adjoint) operator                              
in terms of its eigenvalues.
By \cite[Theorem 3.7]{SB},
\begin{equation}
 \Tr e^{-t\Delta_{\tau,\chi}^\sharp}=\sum_{\lambda_{j}\in\spec(\Delta_{\tau,\chi}^\sharp)}m(\lambda_{j})e^{-t\lambda_{j}},
\end{equation}
where $m(\lambda_{j})$ is in the Definition 4.6.
The kernel function $H^{\tau,\chi}_{t}$ of the integral operator $e^{-t\Delta_{\tau,\chi}^\sharp}$ is a smooth section of 
$(E_{\tau}\otimes E_{\chi})\otimes(E_{\tau}\otimes E_{\chi})^{*}$, 
i.e.,
\begin{equation*}
 H^{\tau,\chi}_{t}\in C^{\infty}(X,(E_{\tau}
 \otimes E_{\chi})\otimes(E_{\tau}\otimes E_{\chi})^{*}).
\end{equation*}
It can be expressed as
\begin{equation*}
 H^{\tau,\chi}_{t}(x,y)=\sum_{\gamma \in \Gamma}\widetilde{H}^{\tau}_{t}(\widetilde{x},\gamma\widetilde{y})\otimes \chi(\gamma)\Id_{V_{\chi}},
\end{equation*}
where $\widetilde{x},\widetilde{y}$ are lifts of $x,y$ to $\widetilde{X}$, respectively, and
$\widetilde{H}^{\tau}_{t}$ is the kernel of $e^{-t\widetilde{\Delta}_{\tau}}$.
By \cite[Proposition 4.1]{M1}, we have the following proposition.
\begin{prop}
Let $E_{\chi}$ be a flat vector bundle over $X=\Gamma\backslash \widetilde{X}$ associated with a finite dimensional complex
representation $\chi\colon\Gamma\rightarrow \GL(V_{\chi})$ of $\Gamma$. Let $\Delta_{\tau,\chi}^{\sharp}$ be the twisted Bochner-Laplace operator
acting on $C^{\infty}(X,E_{\tau}\otimes E_{\chi})$.
Then, 
\begin{align}
 \Tr(e^{-t\Delta_{\tau,\chi}^{\sharp}})=\notag&\sum_{\lambda_{j}\in\spec(\Delta_{\tau,\chi}^\sharp)}m(\lambda_{j})e^{-t\lambda_{j}}\\
&=\sum_{\gamma \in \Gamma}\tr \chi(\gamma)\int_{\Gamma\backslash G}\tr H_{t}^{\tau}(g^{-1}\gamma g)d\dot{g}.
\end{align}
\end{prop}
We proceed further to obtain a better version of the trace formula, analyzing the above identity in orbital integrals.
We group together into the conjugacy classes $[\gamma]$ of $\Gamma$, and
we write separately the conjugacy class of the identity element $e$ to get
\begin{align}
\Tr(e^{-t\Delta_{\tau,\chi}^{\sharp}})=\notag&\dim(V_{\chi})\Vol(X)\tr H_{t}^{\tau}(e)\\
&+\sum_{[\gamma]\neq e}\tr\chi(\gamma)\Vol(\Gamma_{\gamma}\backslash G_{\gamma})\int_{G_{\gamma}\backslash G} \tr H_{t}^{\tau}(g^{-1}\gamma g)d\dot{g},
\end{align}
where $\Gamma_{\gamma}$ and $G_{\gamma}$ are the centralizers of $\gamma$ in $\Gamma$ and $G$, respectively.

We are interested in the zeta functions associated with a geodesic flow
on the bundle $E(\sigma,\chi):=G\times_{\sigma\otimes\chi}(V_\sigma\otimes V_\gamma)\rightarrow S(X)$
as it is explained in section 3.
Let $R(M)^{+}$ and $R(M)^{-}$ be the subspaces of the elements of $R(M)$ that are invariant, respectively not invariant, under the action of the Weyl group $W_{A}$.
More precisely, since the order of the Weyl group $W_{A}$ is two, there is an eigenspace decomposition of $R(M)$ into $R(M)^{+}$ and  $R(M)^{-}$.
The subspaces $R(M)^{\pm}$ correspond to the $(\pm1)$-eigenspaces, respectively.
\begin{prop}
\begin{enumerate}
 \item The map $i^{*}$ is a bijection between $R(K)$ and $R(M)^{+}$
 \item If $\sigma\in R(M)^{-}$, then there exists a unique element $\tau(\sigma)\in \widehat{K}$, 
 with highest weight $\nu_{\tau}=\big((\nu_{1}-\frac{1}{2})e_{1},\ldots,(\nu_{n}-\frac{1}{2})e_{n}\big)$ and
 $\nu_{n}(\sigma)>0$,
 such that (4.41), 
 \begin{equation}
  \sigma-w\sigma=(s^{+}-s^{-})i^{*}(\tau(\sigma)),
 \end{equation}
where $s^{+}, s^{-}$ are the half spin representations of $M$.
More precisely, if $s$ is the spin representation of $K$, then $\tau(\sigma)\otimes s$ splits into
\begin{equation}
 \tau(\sigma)\otimes s=\tau^{+}(\sigma)\oplus\tau^{-}(\sigma)
\end{equation}
 such that
\begin{equation}
 \sigma+w\sigma=i^{*}(\tau^{+}(\sigma)-\tau^{-}(\sigma)),
\end{equation}
with
\begin{equation}
 \tau^{\pm}(\sigma)=\sum_{\substack{\mu\in\{0,1\}^{n}\\c(\mu)=\pm 1}}(-1)^{c(\mu)}\tau_{\nu_{\mu}}(\sigma),
\end{equation}
where $c(\mu):=\sharp \{1\in\mu\}$, $\tau_{\nu_{\mu}}(\sigma)$ is the representation
of $K$ with highest weight $\nu_{\mu}(\sigma)=\nu_{\sigma}-\mu$, and $\nu_{\sigma}$
is given by \emph{(2.5)}.
\end{enumerate}
\end{prop}
\begin{proof}
 See \cite[Proposition 1.1]{BO}.
\end{proof}
Let $\tau_{\sigma}\in R(K)$ with $\tau_{\sigma}:=\tau^{+}(\sigma)-\tau^{-}(\sigma)$.
By Proposition 5.4, there exist unique integers
$m_{\tau}(\sigma)\in\{-1,0,1\}$, which are equal to zero except for finitely many $\tau\in \widehat{K}$,
such that 
\begin{equation}
 \sigma=\sum_{\tau\in\widehat{K}}m_{\tau}(\sigma)i^{*}(\tau);    
 \end{equation}
Then, the locally homogeneous vector bundle $E(\sigma)$ associated with $\tau$ is of the form
\begin{equation}
 E(\sigma)=\bigoplus_{\substack{\tau\in\widehat{K}\\m_{\tau}(\sigma)\neq 0}}E_{\tau},
\end{equation}
where $E_{\tau}$ is the locally homogeneous vector bundle associated with $\tau\in\widehat{K}$.
Therefore, the vector bundle $E(\sigma)$ has a grading $E(\sigma)=E(\sigma)^{+}\oplus E(\sigma)^{-}$.
This grading is defined exactly by the positive or negative sign of $m_{\tau}(\sigma)$.
Let $\widetilde{E}(\sigma)$ be the pullback of $E(\sigma)$ to $\widetilde{X}$.
Then,
\begin{equation*}
 \widetilde{E}(\sigma)=\bigoplus_{\substack{\tau\in\widehat{K}\\m_{\tau}(\sigma)\neq0}}\widetilde{E}_{\tau}.
\end{equation*}
We consider now the lift $\widetilde{\Delta}_{\tau}$ of the Bochner-Laplace operator ${\Delta_{\tau}}$ associated to $\tau\in\widehat{K}$ to $\widetilde{X}$, acting on smooth sections of $\widetilde{E}_{\tau}$.
Recall equation (5.4):
\begin{equation*}
 \widetilde{\Delta}_{\tau}=-R(\Omega)+\lambda_{\tau}\Id.
\end{equation*}
We put
\begin{equation}
 \widetilde{A}_{\tau}:=\widetilde{\Delta}_{\tau}-\lambda_{\tau}\Id.
\end{equation}
Hence, by (5.4) the operator $\widetilde{A}_{\tau}$ acts like $-R(\Omega)$ on the space of smooth sections of $\widetilde{E}_{\tau}$. 
It is an elliptic formally self-adjoint operator of second order. By \cite{Ch}, it is an essentially self-adjoint operator.
Its self-adjoint extension will be also denoted by $\widetilde{A}_{\tau}$.

We use the flat Laplacian $\widetilde{\Delta}^{\sharp}_{\tau,\chi}$ on the universal covering $\widetilde{X}$.
We get then the operator $\widetilde{A}^{\sharp}_{\tau,\chi}$ acting on the space $C^{\infty}(\widetilde{X},\widetilde{E}_{\tau}\otimes\widetilde{E}_{\chi})$.
Since $\widetilde{A}^{\sharp}_{\tau,\chi}$ is induced by the operator $\widetilde{\Delta}^{\sharp}_{\tau,\chi}$, 
it can be locally described as
\begin{equation}
 \widetilde{A}^{\sharp}_{\tau,\chi}=\widetilde{A}_{\tau}\otimes\Id_{V_{\chi}}.
\end{equation}
We pass to $X=\Gamma\backslash \widetilde{X}$. 
We put 
\begin{equation}
 c(\sigma):=-\lvert \rho \rvert^{2}-\lvert\rho_{m}\rvert^{2}+\lvert \nu_{\sigma}+\rho_{m}\rvert^{2},
\end{equation}
where $\nu_{\sigma}$ is the highest weight of $\sigma\in\widehat{M}$ as in (2.5)
and $\rho,\rho_{m}$ are defined by (2.3) and (2.4), respectively.
We define the operator $A_{\chi}^{\sharp}(\sigma)$ acting on $C^{\infty}(X,E(\sigma)\otimes E_{\chi})$ by
\begin{equation}
 A_{\chi}^{\sharp}(\sigma):=\bigoplus_{m_{\tau}(\sigma)\neq 0}A_{\tau,\chi}^{\sharp}+c(\sigma).
\end{equation}
Obviously, $A_{\chi}^{\sharp}(\sigma)$ preserves the grading.
It is an elliptic operator of order two. However, the situation is now different,
because it is not a self-adjoint operator anymore. This property is carried by the operator $A^{\sharp}_{\tau,\chi}$.

We deal first with the corresponding heat semi-group generated by the operator $e^{-tA_{\tau,\chi}^{\sharp}}$. 
Since $A_{\tau,\chi}^{\sharp}$ is induced by $\Delta^{\sharp}_{\tau,\chi}$, it is an integral operator with smooth kernel.
By Proposition 5.3, its trace is given by
\begin{equation}
 \Tr(e^{-tA_{\tau,\chi}^{\sharp}})=\sum_{\gamma \in \Gamma}\tr \chi(\gamma)\int_{\Gamma\backslash G}\tr Q_{t}^{\tau}(g^{-1}\gamma g)d\dot{g},
\end{equation}
where  $Q^{\tau}_{t}\in (\mathcal{C}^{q}(G)\otimes \End(V_{\tau}))^{K\times K}$ is the kernel associated to the operator
$e^{-t\widetilde{A}_{\tau}}$.
We put
\begin{equation*}
 q_{t}^{\tau}=\tr Q_{t}^{\tau}(g).
\end{equation*} 
\begin{align}
q_{t}^{\sigma}=&\sum_{\tau \in \widehat{K}}m_{\tau}(\sigma) q_{t}^{\tau},\\
K(t;\sigma)=&\sum_{\tau \in \widehat{K}}m_{\tau}(\sigma) \Tr(e^{-tA^{\sharp}_{\tau,\chi}}).
\end{align}
We use now the trace formula from \cite[p.177-178]{Wa}, which expands (5.27). We have
 \begin{align}
  K(t;\sigma)=\notag&\dim(V_{\chi})\Vol(X)q_{t}^{\sigma}(e)\\
  &+\frac{1}{2\pi}\sum_{[\gamma]\neq e} \frac{l(\gamma)\tr(\chi(\gamma))}{n_{\Gamma}(\gamma)D(\gamma)}\sum_{\sigma\in\widehat{M}}
  \overline{\tr\sigma(m_{\gamma})}\int_{\R}\Theta_{\sigma,\lambda}(q_{t}^{\sigma})e^{-il(\gamma)\lambda}d\lambda.
\end{align}
We continue analysing the trace formula above in terms of characters. For the identity contribution we have 
\begin{equation}
 (q_{t}^{\sigma})(e)=\sum_{\sigma\in\widehat{M}}\int_{\R}\Theta_{\sigma,\lambda}(q_{t}^{\sigma})P_{\sigma}(i\lambda)d\lambda,
\end{equation}
where $P_{\sigma}(i\lambda)$ denotes the Plancherel polynomial, defined in section 2.
By equation (5.28) we get 
\begin{equation}
 \Theta_{\sigma,\lambda}(q_{t}^{\sigma})=
 \sum_{\tau \in \widehat{K}}m_{\tau}(\sigma) \Theta_{\sigma,\lambda}(q_{t}^{\tau}).
\end{equation}
By Proposition 5.2,
\begin{equation}
 \Theta_{\sigma,\lambda}(q_{t}^{\tau})=e^{-t(-\pi_{\sigma,\lambda}(\Omega))}[\tau\rvert_{M}:\sigma].
\end{equation}
The term $\lambda_{\tau}$ does not occur here, since our operator $A^{\sharp}_{\tau,\chi}$ is induced by the operator $A_{\tau}=\Delta_{\tau}-\lambda_{\tau}\Id$.
We recall also 
\begin{equation}
 \pi_{\sigma,\lambda}(\Omega)=-\lambda^{2}+c(\sigma).
\end{equation}
This is proved in \cite[p.48]{Ar}.\\
Combining equations (5.32), (5.33) and (5.34) we get
\begin{equation}
 \Theta_{\sigma,\lambda}(q_{t}^{\sigma})=\sum_{\tau \in \widehat{K}}m_{\tau}(\sigma)
 e^{-t(\lambda^{2}-c(\sigma))}[\tau\rvert_{M}:\sigma].
\end{equation}
Equivalently, for $\sigma,\sigma'\in \widehat{M}$
\begin{equation*}
  \Theta_{\sigma',\lambda}(q_{t}^{\sigma})=e^{tc(\sigma)}e^{-t\lambda^{2}}\bigg[\sum_{\tau \in \widehat{K}}m_{\tau}(\sigma)i^{*}(\tau):\sigma'\bigg].
\end{equation*}
Hence, by (5.21), 
\begin{align}
& \Theta_{\sigma',\lambda}(q_{t}^{\sigma})=e^{tc(\sigma)}e^{-t\lambda^{2}}.
\end{align}
If we put everything together and insert (5.31), and (5.36) in (5.30), we obtain
  \begin{align*}
  K(t;\sigma)=&e^{tc(\sigma_{k})}\bigg(\dim(V_{\chi})\Vol(X)\int_{\R}e^{-t\lambda^{2}}P_{\sigma}(i\lambda)d\lambda\\
  &+\sum_{[\gamma]\neq[e]} \frac{l(\gamma)}{n_{\Gamma}(\gamma)}L_{sym}(\gamma;\sigma)\frac{e^{-l(\gamma)^{2}/4t}}{(4\pi t)^{1/2}}\bigg);
\end{align*}
where 
\begin{equation}
 L_{sym}(\gamma;\sigma)= \frac{\tr(\chi(\gamma)\otimes\sigma(m_{\gamma}))e^{-|\rho|l(\gamma)}}{\det(\Id-\Ad(m_{\gamma}a_{\gamma})_{\overline{n}})}.
\end{equation}
\newline
By the definition of the operator $A_{\chi}^{\sharp}(\sigma)$ in (5.26), we get the following theorem.

\begin{thm} 
 For every $\sigma \in \widehat{M}$,
 \begin{align} \label{f:trace formula 1}
  \Tr(e^{-tA_{\chi}^{\sharp}(\sigma)})=&\dim(V_{\chi})\Vol(X)\int_{\R}e^{-t\lambda^{2}}P_{\sigma}(i\lambda)d\lambda\\\notag
  &+\sum_{[\gamma]\neq e} \frac{l(\gamma)}{n_{\Gamma}(\gamma)}L_{sym}(\gamma;\sigma)
  \frac{e^{-l(\gamma)^{{2}}/4t}}{(4\pi t)^{1/2}};
 \end{align}
where  $L_{sym}(\gamma;\sigma)$ is as in \emph{(5.37)}.
\end{thm}

\begin{center}
\section{\textnormal{Meromorphic continuation of the zeta functions}}
\end{center}

Let $A$ be a closed linear operator, defined on a dense subspace of $\mathcal{D}(A)$ of a Hilbert space $\mathcal{H}$. Let $a\in\C-\spec(-A)$.
We set $R(a):=(A+a\Id)^{-1}=(A+a)^{-1}$.
Then, the resolvent identity states
\begin{equation*}
 R(a)-R(b)=(b-a)R(a)R(b),
\end{equation*}
for all $a,b\in\C-\spec(-A)$.
The generalized resolvent identity is described in the following lemma.
\begin{lem}
Let $s_{1},\ldots,s_{N}\in\C-\spec(-A)$, $N\in\N$, such that $s_{i}\neq s_{j}$ 
for all $i\neq j$.
Then,
\begin{equation}
\label{generalized resolvent}
 \prod_{i=1}^{N}R(s_{i})=\sum_{i=1}^{N}\bigg(\prod_{\substack{j=1\\ j\neq i}}^{N}\frac{1}{s_{j}-s_{i}}\bigg)R(s_{i}).
\end{equation}
\end{lem}
\begin{proof}
This is proved in \cite[Lemma 3.5]{BO}.
\end{proof}
We will use also the following lemmata.
\begin{lem}
Let $s_{1},\ldots,s_{N}\in\C$, $N\in\N$,
such that $s_{i}\neq s_{j}$ for all $i\neq j$ and let $l=0,1,\ldots, N-2$. Then, we have
\begin{equation}
 \sum_{i=1}^{N}s_{i}^{2l}\bigg(\prod_{\substack{j=1\\ j\neq i}}^{N}\frac{1}{s_{j}^{2}-s_{i}^{2}}\bigg)=0.
\end{equation}
\begin{proof}
 This follows from \cite[Lemma 3.6]{BO}, applied to $s_{i}^{2}$.
\end{proof}
\end{lem}
\begin{lem}
Let $s_{1},\ldots,s_{N}\in\C$, $N\in\N$,
such that $s_{i}\neq s_{j}$ for all $i\neq j$.
Then,
\begin{equation}
 \sum_{i=1}^{N}\bigg(\prod_{\substack{j=1\\ j\neq i}}^{N}\frac{1}{s_{j}^{2}-s_{i}^{2}}\bigg) e^{-ts_{i}^{2}}=O(t^{N-1}),
\end{equation}
as $t\rightarrow 0^{+}$.
\end{lem}
\begin{proof}
We will use the Taylor expansion of the exponential function $e^{-ts_{i}^{2}}$.\\
We have as  $t\rightarrow 0^{+}$
\begin{align*}
  \sum_{i=1}^{N}\bigg(\prod_{\substack{j=1\\ j\neq i}}^{N}\frac{1}{s_{j}^{2}-s_{i}^{2}}\bigg) e^{-ts_{i}^{2}}&=\sum_{k=1}^{N-2}\sum_{i=1}^{N}\frac{(-t)^k}{k!}s_{i}^{2k}\bigg(\prod_{\substack{j=1\\ j\neq i}}^{N}\frac{1}{s_{j}^{2}-s_{i}^{2}}\bigg)+O(t^{N-1})\\
&=\sum_{k=1}^{N-2}\frac{(-t)^k}{k!}\sum_{i=1}^{N} s_{i}^{2k}\bigg(\prod_{\substack{j=1\\ j\neq i}}^{N}\frac{1}{s_{j}^{2}-s_{i}^{2}}\bigg)+O(t^{N-1})=O(t^{N-1}),
\end{align*}
where in the last equality we used Lemma 6.2.
\end{proof}
\begin{lem}
Let $s_{1}\in \C$ such that \emph{Re}$(s_{i}^{2})>0$ for all $i=1,\ldots, N$. Then, the following integral 
\begin{equation}
\int_{0}^{\infty}\int_{\R}\sum_{k=1}^{N}
\bigg(\prod_{\substack{j=1\\ j\neq k}}^{N}\frac{1}{s_{j}^{2}-s_{k}^{2}}\bigg) 
e^{-t(s_{k}^{2}+\lambda^{2})} P_{\sigma}(i\lambda)d\lambda dt
\end{equation}
converges absolutely.
\end{lem}
\begin{proof}
We have as $t\rightarrow\infty$,
\begin{equation}
\int_{\R}\bigg|\sum_{k=1}^{N}\bigg(\prod_{\substack{j=1\\ j\neq k}}^{N}\frac{1}{s_{j}^{2}-
s_{k}^{2}}\bigg) e^{-t(s_{k}^{2}+\lambda^{2})} P_{\sigma}(i\lambda)\bigg| d\lambda=O(e^{-t\epsilon}),
\end{equation}
for some $\epsilon> 0$.\\
We use now the fact the $P(i\lambda)$ is an even polynomial of degree $2n$ (\cite[264-265]{Mi2}).
If we make a change of variables $\lambda'\mapsto \lambda \sqrt{t}$, we get as $t\rightarrow 0^{+}$,
\begin{align}
 \int_{\R}\left|e^{-t\lambda^{2}}P_{\sigma}(i\lambda)\right| d\lambda =O(t^{-d/2}).
\end{align}
Hence, if we combine (6.3) and (6.6) we have that as  $t\rightarrow 0^{+}$,
\begin{equation}
\int_{\R}\bigg|\sum_{k=1}^{N}\bigg(\prod_{\substack{j=1\\ j\neq k}}^{N}\frac{1}{s_{j}^{2}-s_{k}^{2}}
\bigg) e^{-t(s_{k}^{2}+\lambda^{2})} P_{\sigma}(i\lambda)\bigg| d\lambda=O(t^{-d/2+N-1}).
\end{equation}
The assertion follows from (6.5) and (6.7).
\end{proof}
Let $N\in\N$. Let $s_{i},i=1,\ldots,N$ be complex numbers such that $s_{i}\in\C-\spec(-{A^{\sharp}_{\chi}(\sigma)})$.
We consider the resolvent operator
$
 R(s_{i}^{2})=({A^{\sharp}_{\chi}(\sigma)}+s_{i}^{2})^{-1}.
$
We want to obtain the trace class property of the operator
$
\prod_{i=1}^{N}R(s_{i}^{2}).
$
In order to obtain this property, we take sufficient large $N\in\N$,
such that for $N>\frac{d}{2}$,
\begin{equation}
 \Tr(\prod_{i=1}^{N}R(s_{i}^{2}))<\infty.
\end{equation}
We denote the space of pseudodifferential operators of order $k$ by $\psi DO^{k}$.
To prove the trace class property of the operators above, we
observe at first that $\prod_{i=1}^{N}R(s_{i}^{2})\in\psi DO^{-2N}$.

Let $\Delta$ be the Bochner-Laplace operator with respect
to some metric, acting on $C^{\infty}(X, E_{\tau_{\sigma}}\otimes E_{\chi})$.
Then, $\Delta$ is a second-order elliptic differential operator,
which is formally self-adjoint and non-negative, i.e.,
$\Delta\geq 0$. Then, by Weyl's law, we have that 
for $N>\frac{d}{2}$, 
$
 (\Delta+\Id)^{-N}
$
is a trace class operator.
Moreover,
$
 B:=(\Delta+\Id)^{N}\prod_{i=1}^{N}R(s_{i}^{2})
$
is $\psi DO$ of order zero.
Hence, it defines a bounded operator in $L^{2}(X, E_{\tau_{s}(\sigma)}\otimes E_{\chi})$.\\
Thus,
$
 \prod_{i=1}^{N}R(s_{i}^{2})=(\Delta+\Id)^{-N}B
$
is a trace class operator.

We recall here the following expressions of the resolvents. Let $s_{1},\ldots,s_{N}\in\C$ such that 
Re$(s_{i}^{2})>-c$, for all $i=1,\ldots,N$, where $c$ is a real number such that $\spec\big(A_{\chi}^{\sharp}(\sigma)\big)\subset
\{z\in\C\colon \text{Re}(z)>c\}$.
Then,
\begin{align}
(A_{\chi}^{\sharp}(\sigma)+s_{i}^{2})^{-1}&=\int_{0}^{\infty}e^{-ts_{i}^{2}}e^{-t{A_{\chi}^{\sharp}(\sigma)}}dt.
\end{align}
Let  $N\in\N$ with $N>d/2$.
Let $s_{1},\ldots,s_{N}\in\C$ with 
 $s_{i}\neq s_{j}$ for all $i\neq j$ such that 
 Re$(s_{i}^{2})>-C$, for all $i=1,\ldots,N$, where $C$ is a real number such that $\spec\big(A_{\chi}^{\sharp}(\sigma)\big)\subset
\{z\in\C\colon \text{Re}(z)>C\}$.\\
 By  Lemma 6.1 and equation (6.9), we get
\begin{align*}
  \prod_{i=1}^{N}(A_{\chi}^{\sharp}(\sigma)+s_{i}^{2})^{-1}=
\int_{0}^{\infty}\sum_{i=1}^{N}\bigg(\prod_{\substack{j=1\\ j\neq i}}^{N}\frac{1}{s_{j}^{2}-s_{i}^{2}}\bigg)e^{-ts_{i}^{2}}e^{-t{A_{\chi}^{\sharp}(\sigma)}}dt.\\
\end{align*}
Then,
\begin{align*}
  \Tr \prod_{i=1}^{N}(A_{\chi}^{\sharp}(\sigma)+s_{i}^{2})^{-1}=
\int_{0}^{\infty}\sum_{i=1}^{N}\bigg(\prod_{\substack{j=1\\ j\neq i}}^{N}\frac{1}{s_{j}^{2}-s_{i}^{2}}\bigg)e^{-ts_{i}^{2}}\Tr e^{-t{A_{\chi}^{\sharp}(\sigma)}}dt.\\
\end{align*}
We insert now the right-hand side of the trace formula (5.38) for the operator $A_{\chi}^{\sharp}(\sigma)$ and get 
\begin{align*}
\Tr \prod_{i=1}^{N}(A_{\chi}^{\sharp}(\sigma)+&s_{i}^{2})^{-1}=\int_{0}^{\infty}\sum_{i=1}^{N}\bigg(\prod_{\substack{j=1\\ j\neq i}}^{N}\frac{1}{s_{j}^{2}-s_{i}^{2}}\bigg)e^{-ts_{i}^{2}}\\
&\bigg\{\dim(V_{\chi})\Vol(X)\int_{\R}e^{-t\lambda^{2}}P_{\sigma}(i\lambda)d\lambda
  +\sum_{[\gamma]\neq e} \frac{l(\gamma)}{n_{\Gamma}(\gamma)}L_{sym}(\gamma;\sigma)
  \frac{e^{-l(\gamma)^{{2}}/(4t)}}{(4\pi t)^{1/2}}\bigg\}dt.
\end{align*}
Hence,
\begin{align}
\Tr \prod_{i=1}^{N}(A_{\chi}^{\sharp}(\sigma)+s_{i}^{2})^{-1}\notag&=\dim(V_{\chi})\Vol(X)\int_{0}^{\infty}\int_{\R}\sum_{i=1}^{N}\bigg(\prod_{\substack{j=1\\ j\neq i}}^{N}\frac{1}{s_{j}^{2}-s_{i}^{2}}\bigg)
e^{-ts_{i}^{2}}e^{-t\lambda^{2}}P_{\sigma}(i\lambda)d\lambda dt\\
&+\sum_{i=1}^{N}\int_{0}^{\infty}\bigg(\prod_{\substack{j=1\\ j\neq i}}^{N}\frac{1}{s_{j}^{2}-s_{i}^{2}}\bigg)e^{-ts_{i}^{2}}\bigg\{\sum_{[\gamma]\neq[e]} \frac{l(\gamma)}{n_{\Gamma}(\gamma)}L_{sym}(\gamma;\sigma)
\frac{e^{-l(\gamma)^{2}/(4t)}}{(4\pi t)^{1/2}}\bigg\}dt.
\end{align}
The first sum in the right-hand side of (6.10), which involves the double integral can be explicitly calculated.
We set 
\begin{equation*}
 I:=\int_{0}^{\infty}\int_{\R}\sum_{i=1}^{N}\bigg(\prod_{\substack{j=1\\ j\neq i}}^{N}\frac{1}{s_{j}^{2}-s_{i}^{2}}\bigg)
e^{-ts_{i}^{2}}e^{-t\lambda^{2}}P_{\sigma}(i\lambda)d\lambda dt.
\end{equation*}
By Lemma 6.4, we can interchange the order of the integration and get
\begin{align*}
I=\int_{\R}\int_{0}^{\infty}\sum_{i=1}^{N}\bigg(\prod_{\substack{j=1\\ j\neq i}}^{N}\frac{1}{s_{j}^{2}-s_{i}^{2}}\bigg)
e^{-t(s_{i}^{2}+\lambda^{2})}P_{\sigma}(i\lambda)dtd\lambda.\\
\end{align*}
By \cite[Lemma 3.5]{BO} and since $P_{\sigma}$ is an even polynomial, we obtain the following convergent integral
\begin{equation*}
 I=\int_{\R}\sum_{i=1}^{N}\bigg(\prod_{\substack{j=1\\ j\neq i}}^{N}\frac{1}{s_{j}^{2}-s_{i}^{2}}\bigg)
\bigg(\frac{1}{\lambda^{2}+s_{i}^{2}}\bigg)P_{\sigma}(i\lambda)d\lambda.
\end{equation*}
Using the Cauchy integral formula we have
\begin{align}
I=\sum_{i=1}^{N}\bigg(\prod_{\substack{j=1\\ j\neq i}}^{N}\frac{1}{s_{j}^{2}-s_{i}^{2}}\bigg)
\frac{\pi}{s_{i}}P_{\sigma}(s_{i}).
\end{align}
For the second sum in the right-hand side of (6.10) we use the formula (see \cite[p.146,(27)]{Er})
\begin{equation}
 \int_{0}^{\infty}e^{-ts^{2}} \frac{e^{-l(\gamma)^{{2}}/(4t)}}{(4\pi t)^{1/2}}dt=\frac{1}{2s}e^{-sl(\gamma)}.
\end{equation} 
Hence, equation (6.10) becomes by (6.11) and (6.12)
\begin{align}
\Tr \prod_{i=1}^{N}(A_{\chi}^{\sharp}(\sigma)+s_{i}^{2})^{-1}\notag&=\sum_{i=1}^{N}\bigg(\prod_{\substack{j=1\\ j\neq i}}^{N}\frac{1}{s_{j}^{2}-s_{i}^{2}}\bigg)\frac{\pi}{s_{i}} \dim(V_{\chi})\Vol(X)P(s_{i})\\
&+\sum_{i=1}^{N}\frac{1}{2s_{i}}\bigg(\prod_{\substack{j=1\\ j\neq i}}^{N}\frac{1}{s_{j}^{2}-s_{i}^{2}}\bigg)\sum_{[\gamma]\neq e} \frac{l(\gamma)}{n_{\Gamma}(\gamma)}L_{sym}(\gamma;\sigma)
e^{-s_{i}l(\gamma)}.
\end{align}
By Lemma 3.6, the sum over the conjugacy classes $[\gamma]$ of $\Gamma$ in the right-hand side of (6.13) is equal to the 
logarithmic derivative  $L(s_{i})=\frac{d}{ds}\log(Z(s_{i};\sigma,\chi)$ of the Selberg zeta function.
Hence,
\begin{align}
\Tr \prod_{i=1}^{N}(A_{\chi}^{\sharp}(\sigma)+s_{i}^{2})^{-1}\notag&=\sum_{i=1}^{N}\bigg(\prod_{\substack{j=1\\ j\neq i}}^{N}\frac{1}{s_{j}^{2}-s_{i}^{2}}\bigg)\frac{\pi}{s_{i}} \dim(V_{\chi})\Vol(X)P(s_{i})\\
&+\sum_{i=1}^{N}\frac{1}{2s_{i}}\bigg(\prod_{\substack{j=1\\ j\neq i}}^{N}\frac{1}{s_{j}^{2}-s_{i}^{2}}\bigg)L(s_{i}).
\end{align}
In the Theorem 6.5 below, we choose the branch of the square roots
of the complex numbers $t_{k}$, whose real part is positive.
In addition, if $t_{k}$ are negative real numbers, we choose the branch of the square roots,
whose imaginary part is positive.
\begin{thm}
The Selberg zeta function $Z(s;\sigma,\chi)$ admits a meromorphic continuation to the whole complex plane $\C$. The set of the singularities
equals $\{s_{k}^{\pm}=\pm i \sqrt{t_{k}}:t_{k}\in \spec(A^{\sharp}_{\chi}(\sigma)), k\in\N\}$.
The orders of the singularities are equal to $m(t_{k})$, where $m(t_{k})\in\N$ denotes the algebraic multiplicity of the eigenvalue $t_{k}$.
For $t_{0}=0$, the order of the singularity $s_{0}$ is equal to $2m(0)$.
\end{thm}
\begin{proof}
By (6.1) and (6.14) we get
\begin{align*}
\Tr\sum_{i=1}^{N}\bigg(\prod_{\substack{j=1\\ j\neq i}}^{N}\frac{1}{s_{j}^{2}-s_{i}^{2}}\bigg)(A_{\chi}^{\sharp}(\sigma)+s_{i}^{2})^{-1}&=\sum_{i=1}^{N}\bigg(\prod_{\substack{j=1\\ j\neq i}}^{N}\frac{1}{s_{j}^{2}-s_{i}^{2}}\bigg)\frac{\pi}{s_{i}} \dim(V_{\chi})\Vol(X)P(s_{i})\\\notag
&+\sum_{i=1}^{N}\frac{1}{2s_{i}}\bigg(\prod_{\substack{j=1\\ j\neq i}}^{N}\frac{1}{s_{j}^{2}-s_{i}^{2}}\bigg)L(s_{i}).
\end{align*}
Equivalently,
\begin{align}
\label{eq selberg mer}
\sum_{i=1}^{N}\bigg(\prod_{\substack{j=1\\ j\neq i}}^{N}\frac{1}{s_{j}^{2}-s_{i}^{2}}\bigg)\frac{1}{2s_{i}}L(s_{i})=\notag&-\sum_{i=1}^{N}\bigg(\prod_{\substack{j=1\\ j\neq i}}^{N}\frac{1}{s_{j}^{2}-s_{i}^{2}}\bigg)\frac{\pi}{s_{i}}\dim(V_{\chi})\Vol(X)P(s_{i})\\\notag
&+\sum_{t_{k}}\sum_{i=1}^{N}\bigg(\prod_{\substack{j=1\\ j\neq i}}^{N}\frac{1}{s_{j}^{2}-s_{i}^{2}}\bigg)\frac{m(t_{k})}{t_{k}+s_{i}^{2}},\\
\end{align}
If we multiply equation (6.15) by $2s_{1}$, we obtain 
\begin{align}
\notag\sum_{i=1}^{N}\bigg(\prod_{\substack{j=1\\ j\neq i}}^{N}\frac{1}{s_{j}^{2}-s_{i}^{2}}\bigg)\frac{s_{1}}{s_{i}}L(s_{i})\notag=&-\sum_{i=1}^{N}\bigg(\prod_{\substack{j=1\\ j\neq i}}^{N}\frac{1}{s_{j}^{2}-s_{i}^{2}}\bigg)\frac{2s_{1}}{s_{i}}\pi\dim(V_{\chi})\Vol(X)P(s_{i})\\\notag
&+\sum_{t_{k}}\sum_{i=1}^{N}\bigg(\prod_{\substack{j=1\\ j\neq i}}^{N}\frac{1}{s_{j}^{2}-s_{i}^{2}}\bigg)2s_{1}\frac{m(t_{k})}{t_{k}+s_{i}^{2}}.\\
\end{align}
Let $ \varPsi(s_{1},\ldots,s_{N})$ be the function of the complex numbers $s_{1},\ldots,s_{N}$, defined by
\begin{equation*}
 \varPsi(s_{1},\ldots,s_{N}):=\sum_{i=1}^{N}\bigg(\prod_{\substack{j=1\\ j\neq i}}^{N}\frac{1}{s_{j}^{2}-s_{i}^{2}}\bigg)\frac{s_{1}}{s_{i}}L(s_{i}).
\end{equation*}
We fix the complex numbers $s_{i}, i=2,\ldots,N$ with $s_{i}\neq s_{j}$ for $i,j=2,\ldots, N$
and let the complex number $s=s_{1}$ vary.\\
Put
\begin{equation*}
 \varPsi(s,\ldots,s_{N})=\varPsi(s).
\end{equation*}
Then, equation (6.16) becomes
\begin{align}
 \notag\varPsi(s)=&-\sum_{i=1}^{N}\bigg(\prod_{\substack{j=1\\ j\neq i}}^{N}\frac{1}{s_{j}^{2}-s_{i}^{2}}\bigg)\frac{2s}{s_{i}}\pi\dim(V_{\chi})\Vol(X)P(s_{i})\\\notag
&\notag+\sum_{t_{k}}\sum_{i=1}^{N}\bigg(\prod_{\substack{j=1\\ j\neq i}}^{N}\frac{1}{s_{j}^{2}-s_{i}^{2}}\bigg)2s\frac{m(t_{k})}{t_{k}+s_{i}^{2}},\\
\end{align}
where $s_{1}=s$.
The term that contains the logarithmic derivative $L(s)$ in $\varPsi(s)$ is of the form
\begin{equation*}
 \bigg(\prod_{j=2}^{N}\frac{1}{s_{j}^{2}-s^{2}}\bigg)L(s).
\end{equation*}
The term of
\begin{equation*}
 \sum_{t_{k}}\sum_{i=1}^{N}\bigg(\prod_{\substack{j=1\\ j\neq i}}^{N}\frac{1}{s_{j}^{2}-s_{i}^{2}}\bigg)2s\frac{m(t_{k})}{t_{k}+s_{i}^{2}},
\end{equation*}
which is singular at $\pm i\sqrt{t_{k}}$, $k\in\N$ is 
\begin{equation*}
 \bigg(\prod_{j=2}^{N}\frac{1}{s_{j}^{2}-s^{2}}\bigg)2s\frac{m(t_{k})}{t_{k}+s^{2}}.
\end{equation*}
We multiply both sides of the 
equality (6.17) by
\begin{equation*}
 \prod_{j=2}^N(s_j^2-s^2).
\end{equation*}
Then, the residues of $L(s)$ at the points $\pm i\sqrt{t_{k}}$
are $m(t_{k})$, for $k\neq0$, and $2m(0)$ for $k=0$.

By (3.20), $L(s)$ decreases exponentially as Re$(s)\rightarrow \infty$. Hence, the integral
\begin{equation*}
 \int_{s}^{\infty}L(w)dw
\end{equation*}
over a path connecting $s$ and infinity is well defined and 
\begin{equation}
 \log Z(s;\sigma,\chi)=-\int_{s}^{\infty}L(w)dw.
\end{equation}
The integral above depends on the choice of the path, because $L(s)$
has singularities.
Since the residues of the singularities 
are integers, it follows that the 
exponential of the integral in the right-hand side of (6.18)
is independent of the choice of the path. The meromorphic continuation of the 
Selberg zeta function $Z(s;\sigma,\chi)$ to the whole complex plane follows.

\end{proof}

In view of the previous results, we will use the meromorphic continuation of the Selberg zeta function to obtain 
the meromorphic continuation of the Ruelle zeta function to the whole complex plane $\C$ .

Following the analysis in \cite[p. 93-94]{BO}, we
consider the identification $\mathfrak{a}_{\C}^{*}\cong \C\ni \lambda$. 
Let $\alpha>0$ be the unique positive
root of the system $(\mathfrak{g},\mathfrak{a})$. Let $\lambda\colon A\to \C^\times$ be the character,
defined by $\lambda(a)=e^{\alpha(\log a)}$. 

Let $\mathfrak{n}_{\C}$ be the complexification of the Lie algebra $\mathfrak{n}$.
Let $\nu_{p}:=\Lambda^{p}\Ad_{\mathfrak{n}_{\C}}(MA)$ be the representation of $MA$ in $\Lambda^{p}\mathfrak{n}_{\C}$ given by the $p$-th exterior power of the adjoint representation:
\begin{equation*}
 \nu_{p}:=\Lambda^{p}\Ad_{\mathfrak{n}_{\C}}\colon MA\rightarrow \GL(\Lambda^{p}\mathfrak{n}_{\C}),\quad p=0,1,\ldots,d-1.
\end{equation*}
For $p=0,1,\ldots,d-1$, let $J_{p}\subset\{(\psi_{p},\lambda)\colon\psi_{p}\in\widehat{M},\lambda\in\C\}$ be the subset consisting of all pairs of unitary irreducible representations of $M$ 
and one dimensional representations of $A$ such that, as $MA$-modules,
the representations $\nu_{p}$ decompose as
\begin{equation*}
 \Lambda^{p}\mathfrak{n}_{\C}=\bigoplus_{(\psi_{p},\lambda)\in J_{p}}V_{\psi_{p}}\otimes \C_{\lambda},
\end{equation*}
where $\C_{\lambda}\cong \C$ denotes the representation space of $\lambda$.

For $\sigma\in \widehat{M}$ we define  
\begin{equation}
 Z_{p}(s;\sigma,\chi):=\prod_{(\psi_{p},\lambda)\in J_{p}}Z(s+\rho-\lambda;\psi_{p}\otimes\sigma,\chi).
\end{equation}
We have then the following theorem, which gives a representation of $R(s;\sigma,\chi)$ as a product of $Z_{p}(s;\sigma,\chi)$ over $p$.
\begin{thm}
Let $\sigma\in\widehat{M}$. Then the Ruelle zeta function has the representation
\begin{equation}
 R(s;\sigma,\chi)=\prod_{p=0}^{d-1} Z_{p}(s;\sigma,\chi)^{(-1)^{p}}.
\end{equation}
\end{thm}
\begin{proof}
By (3.14), we have
\begin{equation}
 \log Z(s+\rho-\lambda;\psi_{p}\otimes\sigma,\chi)=-\sum_{[\gamma]\neq{e}}\frac{1}{n_{\Gamma}(\gamma)}
\tr(\chi(\gamma)\otimes\sigma(m_\gamma))\frac{e^{(-s-2\rho+\lambda)l(\gamma)}\tr(\psi_{p}(m))}{\det(1-\Ad(m_\gamma a_\gamma)|_{\overline{\mathfrak{n}}})}.
\end{equation}
We use now the fact that
\begin{equation}
 {\det(1-\Ad(m_\gamma a_\gamma)|_{\overline{\mathfrak{n}}})}=(-1)^{d-1}a_{\gamma}^{-2\rho}{\det(1-\Ad(m_\gamma a_\gamma)|_{{\mathfrak{n}}})}.
\end{equation}
Hence, if we insert (6.22) in (6.21), we get
\begin{equation}
 \log Z(s+\rho-\lambda;\psi_{p}\otimes\sigma,\chi)=(-1)^{d}\sum_{[\gamma]\neq{e}}\frac{1}{n_{\Gamma}(\gamma)}
\tr(\chi(\gamma)\otimes\sigma(m_\gamma))\frac{e^{-sl(\gamma)} e^{\lambda l(\gamma)}\tr(\psi_{p}(m))}{\det(1-\Ad(m_\gamma a_\gamma)|_{{\mathfrak{n}}})}.
\end{equation}
We have
\begin{align}
\log \prod_{p=0}^{d-1}Z_{p}(s;\sigma,\chi)^{(-1)^{p}}=\notag&\sum_{p=0}^{d-1}\log Z_{p}(s;\sigma,\chi)^{(-1)^{p}}\\\notag
&=\sum_{p=0}^{d-1}{(-1)^{p}}\log Z_{p}(s;\sigma,\chi)\\\notag
&=\sum_{p=0}^{d-1}(-1)^{p}\log\prod_{(\psi_{p},\lambda)\in J_{p}}Z(s+\rho-\lambda;\psi_{p}\otimes\sigma,\chi)\\\notag
&=\sum_{p=0}^{d-1}(-1)^{p} \sum_{(\psi_{p},\lambda)\in J_{p}}\log \bigg(Z(s+\rho-\lambda;\psi_{p}\otimes\sigma,\chi)\bigg)\\\notag
&=\sum_{p=0}^{d-1}(-1)^{p}\bigg((-1)^{d}\sum_{[\gamma]\neq{e}}\frac{1}{n_{\Gamma}(\gamma)}\tr(\chi(\gamma)\otimes\sigma(m_\gamma))e^{-sl(\gamma)}\\\notag
&\frac{\sum_{(\psi_{p},\lambda)\in J_{p}} e^{\lambda l(\gamma)}\tr(\psi_{p}(m))}{\det(1-\Ad(m_\gamma a_\gamma)|_{{\mathfrak{n}}})}\bigg ),\\
\end{align}
where in the third line we used the definition (6.19), and  in the last line, equation (6.23).
We recall now the that for any endomorphism of a finite dimensional vector space $W$, we have
\begin{equation*}
 \det(\Id_{W}-W)=\sum_{p=0}^{\infty}(-1)^{p}\tr(\Lambda^{p}W).                                                                  
\end{equation*}
If we apply this identity to $\Ad(m_\gamma a_\gamma)|_{{\mathfrak{n}}}$ we get
\begin{equation*}
\sum_{p=0}^{d-1}(-1)^{p} \frac{\sum_{(\psi_{p},\lambda)\in J_{p}} e^{\lambda l(\gamma)}\tr(\psi_{p}(m))}{\det(1-\Ad(m_\gamma a_\gamma)_{{\mathfrak{n}}})}=1.
\end{equation*}
Hence, by (6.24) we have
\begin{align*}
\log \prod_{p=0}^{d-1}Z_{p}(s;\sigma,\chi)^{(-1)^{p}}&=(-1)^{d}\sum_{[\gamma]\neq{e}}\frac{1}{n_{\Gamma}(\gamma)}\tr(\chi(\gamma)\otimes\sigma(m_\gamma)){e^{-sl(\gamma)}}\\
&=\log R(s;\sigma,\chi),
\end{align*}
where in the last line we used by equation (3.17) in the proof of Proposition 3.5.
\end{proof}
\begin{thm}
For every $\sigma\in\widehat{M}$, the Ruelle zeta function $R(s;\sigma,\chi)$ admits a meromorphic continuation to the whole complex plane $\C$.
\end{thm}
\begin{proof}
 The assertion follows from  Theorem 6.5 together with Theorem 6.6.
\end{proof}
\bibliographystyle{amsalpha}
\bibliography{references3}
UNIVERSIT\"{A}T BONN, 		
MATHEMATISCHES INSTITUT, 	
ENDENICHER ALLE 60, 		
D-53115,			
GERMANY					\\
\textit{E-mail address}: xspilioti@math.uni-bonn.de

\end{document}